\documentclass[11pt,leqno]{article}

\usepackage{epsfig}
\usepackage[utf8]{inputenc}

  \usepackage{amsmath}
\usepackage{amsthm}
\usepackage{mathtools}
\mathtoolsset{showonlyrefs}

\usepackage[T1]{fontenc}

\usepackage{amssymb}

\newtheorem{theorem}{Theorem}

\newtheorem{proposition}{Proposition}

\newtheorem{lemma}{Lemma}
\newtheorem{remark}{Remark}
\newtheorem{corollary}{Corollary}

\newtheorem{definition}{Definition}
\newcommand{\eps}{\varepsilon}

\newcommand{\C}{\mathcal{C}}
\newcommand{\A}{\mathcal{A}}

\newcommand{\D}{\mathcal{D}}

\newcommand{\G}{\mathcal G}

\newcommand\CC{\hbox{C\kern -.58em {\raise .54ex \hbox
			{$\scriptscriptstyle |$}}
		\kern-.55em {\raise .53ex \hbox{$\scriptscriptstyle |$}} }}
\newcommand\qd{\hfill$\sqcap\kern-8.0pt\hbox{$\sqcup$}$}
\newcommand\NN{\hbox{I\kern-.2em\hbox{N}}}
\newcommand\nn{\hbox{I\kern-.2em\hbox{N}}}
\newcommand\RR{I\!\!R}
\newcommand\sRR{{\sl \hbox{I\kern-.2em\hbox{R}}}}
\newcommand\QQ{\hbox{I\kern-.53em\hbox{Q}}}

\newcommand\sign{\hbox{sign}}

\newcommand\signp{{\hbox{sign}^+}}
\newcommand\so{\hbox{sign}_0}

\newcommand\spo{\hbox{sign}_0^+}

\newcommand\heps{{\mathcal H}_\sigma}
\newcommand\hepsp{{\mathcal H}_\sigma^+}

\newcommand\R{\displaystyle {\mathcal R}}

\usepackage{stackengine}
\usepackage{scalerel}
\usepackage{xcolor,amssymb}

\everymath{\displaystyle}

\begin{document}
 
\title{$L^1-$Theory for  Incompressible Limit  of Reaction-Diffusion  \\ Porous   
Medium	Flow  with Linear Drift}

 \author{Noureddine Igbida \thanks{Institut de recherche XLIM-DMI, UMR-CNRS 7252, Faculté des Sciences et Techniques,  Université de Limoges, France. E-mail : {\tt noureddine.igbida@unilim.fr} . }}  


\date{\today}
 
 \maketitle

\begin{abstract}
Our aim is to study existence, uniqueness and the limit, as $m\to\infty,$ 	of the  solution of reaction-diffusion porous medium equation with linear drift   $\displaystyle\partial_t u  -\Delta u^m +\nabla \cdot (u  \: V)=g(t,x,u) $  in bounded domain with Dirichlet boundary condition.  We treat the problem   without any sign restriction on the solution with  an outpointing  vector field $V$ on the boundary and a general source term $g$ (including the continuous Lipschitz case). By means of new $BV_{loc}$ estimates in bounded domain, under   reasonably sharp Sobolev assumptions on $V$, we show uniform $L^1-$convergence  towards   the solution of  reaction-diffusion  Hele-Shaw flow  with  linear drift.  \\

\textbf{MSC2020 database} : 35A01, 35A02, 35B20, 35B35

\end{abstract}

\section{Introduction and main results}

\subsection{Introduction}
 
Let $\Omega\subset\RR^N$ be a bounded   open set with regular boundary $\displaystyle \partial\Omega =:\Gamma .$   
Our aim here is to study the limit, as $m\to\infty,$ of the equation 
	\begin{equation}\label{eq0}
		\displaystyle \frac{\partial u }{\partial t}  -\Delta u^m +\nabla \cdot (u  \: V)=g(t,x,u) \quad   \hbox{ in } Q:= (0,T)\times \Omega ,
 	\end{equation}
  where  the expression $r^m$  denotes  $\vert r\vert^{m-1}r,$ for any $r\in \RR,$ $1<m<\infty ,$   $V\: :\: \Omega\to \RR^N$ is a given vector field and $g\: :\:  Q\times \RR\to \RR$ is a Carathéodory application.  

\medskip
   There is a huge literature on   qualitative and quantitative studies of \eqref{eq0} in the case where  $V\equiv 0$. We refer the reader to the book \cite{Vbook}   for a thoroughgoing survey of results as well as  corresponding   literature.  In the case  $V\not\equiv  0 , $ the PDE is  a nonlinear version of Fokker-Planck equation   of porous-media type.  This kind of evolutionary problem has gained a significant attention in recent years. It   arises  mainly in biological applications and in the theory of population dynamics. Here  $u$ represents density of  agents traveling  following a vector field $V$ and subject to some local  random motion ; i.e. random exchanges at  a microscopic level between the agent  at a given  position and neighborhoods positions  (on can see for instance  \cite{BeHi,MRS1,MRS2,MRSV,MRSV,PQV1,Ca,Di} and the refs therein for more complement on the applications and motivations).

 Despite the broad  results  on  nonlinear diffusion-transport     PDE (cf. \cite{AltLu,AnIg1,AnIg2,Ca,IgUr,Ot}, see also the expository paper \cite{AnIgSurvey} for a complete list),  the structure of \eqref{eq0}  where the drift   depends  linearly on the density  excludes this class definitely  out the scope  of the current literature.     As far as we know,  the study of existence and uniqueness of weak solution \eqref{eq0} has been   investigated only   for the case of one population in the conservative case leading to one phase    problem in $\RR^N$ or else in bounded domain with    Neumann boundary condition (cf.  \cite{BeHi,Di} in the case $V=\nabla \varphi$  with  reasonable assumptions on the  potential $\varphi$  and $g\equiv 0$). One can see also  \cite{KiLei}  for the study in the framework of viscosity solutions.     Asymptotic convergence to equilibrium  is shown in \cite{BeHi} and  \cite{CaJuMaTo} when  $\varphi$ is convex.   For the regularity of the solution one can see \cite{KZ} and the references therein.

In this paper, we focus chiefly on the case of bounded domain with Dirichlet boundary condition   to study existence and uniqueness of a weak solution of the general formulation  \eqref{eq0} as well as its  limit, as $m\to\infty,$    under general reasonable assumptions on $g$ and $V.$     See that the diffusion term $-\Delta u^m $ may be written as $-\nabla \cdot \left (u\: \nabla \frac{m}{ m-1 } \vert u\vert^{m-1}\right),$ so that the exponent  $m >1$ manage in some sense  the mobility of the agent through a  ''mobility potential''   given by  $\frac{m}{m-1}\vert u\vert ^{m-1}.$   For large $m,$  this term becomes  
   $$ \frac{m}{m-1}\vert u\vert ^{m-1} \approx \left\{ \begin{array}{ll}
   	0\quad  & \hbox{ if } \vert u\vert <1\\
   	+\infty & \hbox{ if } \vert u\vert >1  .
   \end{array} \right.  $$
This formal analysis setup that  the limiting   density $u $  is restrained  to satisfy  $\vert u\vert \leq 1$ within  two   main phases : the so called congestion phase which corresponds to $[\vert u\vert =1]$ and a free one corresponding to $\vert u\vert >0.$ 
More precisely, the limiting   PDE system coincides at least  formally with  a density constrained diffusion  equation  with a linear drift 
 \begin{equation} \label{pdetypehs}
 	\left.  \begin{array}{l}
 		\displaystyle \frac{\partial u }{\partial t} -\nabla\cdot( u\:  \nabla \:  \vert p\vert )  +\nabla \cdot (u  \: V)=g(t,x,u) \\    u\in \sign(p)  
 	\end{array}\right\} 	 \quad  \hbox{ in } Q,
 \end{equation}   
 where   we denote by   $\sign$     the maximal monotone graph given by  
 $$\sign(r)= \left\{ \begin{array}{ll}
 	\displaystyle 1 &\hbox{ for  }r>0\\  
 	\displaystyle [-1,1]\quad & \hbox{ for } r=0\\
 	\displaystyle -1 &\hbox{ for  }r<0. \end{array}\right. \quad   \hbox{ for  }r\in \RR.$$  
 See that  $\nabla\cdot( u\:  \nabla \:  \vert p\vert )   =\Delta p,$ so that \eqref{pdetypehs}  is a    reaction-diffusion system  of Hele-Shaw type   with a linear drift.  The emergence of density constraint $\vert u\vert \leq 1$      is closely  connected to   the microscopic non-overlapping constraint between the agents in the limiting case.  The complementary condition between the density $u$ and the    limiting mobility potential    $p$  typically allows to  describe  the  motion of  congested zones definitely characterized by $[p\neq 0].$      This equation appears in   pedestrian flow (cf. \cite{MRS1}) and in  biological applications (cf. \cite{CaCrYa} and the references therein).   The study of the problem  without any sign restriction on the solution  enables especially to cover mathematical models of two-species  in inter-actions and occupying the same habitat,  like  diffusion-aggregation models. In this cases,  $\rho$ represents through its positive and negative parts  the densities of each specie respectively. The source term $g$ models reaction phenomena connected  to agent  supply  in biological models. This happens in particular when one deals with reaction diffusion system coupling the equations \eqref{eq0} or  \eqref{pdetypehs} 
with other PDE. As to the  boundary condition,   homogeneous  Dirichlet one is connected to the possibility of mobility through  the boundary (exits) without any charge.  One can see \cite{EIG} for other possibilities of boundary condition and their interpretation.

  In this paper, we give the proofs of existence, uniqueness  and    convergence process to   Hele-Shaw flow with linear drift in  the general context  of $L^1-$theory for nonlinear PDE. The approach  differs quite significantly from other recent papers which treats the problem   in $\RR^N$ (cf. \cite{	Noemi,NoDePe,KiPoWo}) in the one phase case by  using mainly classical Aronson-Bénilan estimate (cf. \cite{ArBe}) for nonnegative solutions of porous medium equation. 
  
 In particular, our approach enables to give answers and evidence to many questions left open in many papers dedicated to this subject. Actually, we treat the problem   without any sign restriction on the solution in a bounded domain with Dirichlet boundary condition, low regularity on $V,$  and general source term $g.$  Moreover, the approach   offers many supply for the treatment of  
 the challenging    case of non compatible initial data ; i.e. the case where $\Vert u_0\Vert >1$.  This will be treated separately   in  the forthcoming work    \cite{Igpmsing}.

%
 
  \subsection{Historical notes}
 The  study of the incompressible limit of  \eqref{pm} received a lot of attention since its interest for the applications and for the description of  constrained nonlinear flow.   The problem is well understood by now in the case where $V$ and $g$ vanish (see for instance  \cite{BeCr3}  and \cite{BeBoHe}).     One can see also \cite{Igshaw} for non-homogeneous   Neumann boundary condition and \cite{GQ} for non-homogeneous   Dirichlet one.  In the case where $V\equiv 0$ and $g\not\equiv 0,$   it is know   (see   \cite{BeIgsing} for Dirichlet boundary boundary condition and \cite{BeIgNeumann} for Neuman boundary condition)   that the solution of the problem  
 \begin{equation}\label{eq1}
 	\displaystyle \frac{\partial u }{\partial t}  -\Delta u^m  =g(.,u) \quad   \hbox{ in } Q  ,
 \end{equation}
 converges, as $m\to\infty,$  to the solution of the so called Hele-Shaw problem 
 \begin{equation}  \label{hs1}
 	\left.  \begin{array}{l}
 		\displaystyle \frac{\partial u }{\partial t} -\Delta p=g(.,u)  \\    u\in \sign(p)  
 	\end{array}\right\} 	 \quad  \hbox{ in } Q.
 \end{equation}   
  The convergence holds to be true in $\C([0,T),L^1(\Omega))$ in the case where $\vert u_0\vert \leq 1,$ a.e. in $\Omega,$ otherwise  it holds in $\C((0,T),L^1(\Omega))$ and a boundary layer appears for $t=0.$ This boundary layer is given by the plateau-like function refereed  to as ‘mesa’,  and it is given by the limit, as $m\to\infty,$  of the solution of    homogeneous porous medium equation  
 \begin{equation} \label{pm}
 	\displaystyle \frac{\partial u }{\partial t} =\Delta u^m   \quad   \hbox{ in } Q. 
 \end{equation}    Yet, one needs to be careful with the special case of Neumann boundary condition since, in this case the limiting problem  \eqref{hs1} could be ill posed.  With respect to the assumptions on  $g,$  the  limiting problem exhibits an extra phase  to be mixed with  the  Hele-Shaw phase (see \cite{BeIgNeumann} for more details).   
 Other variations of reaction term    have been proposed  in  recent years together
 with the analysis of their incompressible limit  (see for instance  \cite{PQV1,NoPe,DiSc,KiPo} and the references therein). 
 The recent work \cite{GuKiMe}  treats the particular case of linear reaction term with a special focus on the limit of the so called associated pressure   $p:= \frac{m}{m+1} u^{m-1}$, furthermore the authors seem to be altogether not aware of the general works \cite{BeIgsing,BeIgNeumann}.

\medskip 
The treatment of the case where  $V\not\equiv 0,$ leads to the formal   reaction-diffusion dynamic of Hele-Shaw type   with a linear drift ; i.e. 
\begin{equation} \label{hsg}
	\left.  \begin{array}{l}
		\displaystyle \frac{\partial u }{\partial t} -\Delta p +\nabla \cdot (u  \: V)=g(t,x,u) \\    u\in \sign(p)  
	\end{array}\right\} 	 \quad  \hbox{ in } (0,T)\times \Omega.
\end{equation} 
The problem  was studied first in \cite{BeIgconv} when  $g\equiv 0$ and the   drift term is  of the type $\nabla \cdot F(u),$ with  $F\: :\: \RR\to \RR^N$   a Lipschitz continuous function  (this corresponds particularly  to    space-independent drift).  In   \cite{BeIgconv}, it is proven that   $L^1(\Omega)$-compactness result remains to be true  uniformly in $t.$ Moreover, the limiting problem here is simply the transport equation 
\begin{equation}\label{transport0}
\partial_t u+\nabla \cdot F(u)=0.
\end{equation} 
The Hele-Shaw flow desappears  since the nature of the transport term (incompressible)  in  \eqref{transport0} compel the solution to be  less than $1,$ and then $p\equiv 0.$    Then, in \cite{KiPoWo} the authors studied the case of  space dependent drift and  reaction terms both linear  and  regular in $\Omega=\RR^N.$ Assuming  a monotonicity  on   $ V,$   and using the notion  of viscosity solutions, the authors  study the limit of nonnegative solution, as $m\to\infty$.   The benefit of this  approach is its ability  to cover accurately the free boundary view of the limiting problem (particularly the dynamic of the so called congestion region $[p>0]$), as well as the rate of convergence.     Using  a weak (distributional) interpretation of the solution the same problem was studied recently  in \cite{Noemi}  with a variant of reaction term $g$ in $\Omega=\RR^N.$ 
Using  a blend of recently developed tools  on Aronson-Bénilan regularizing effect as well as sophisticated $L^p-$regularity of the  pressure gradient the authors studied the incompressible limit   in the case of nonnegative compatible initial data  and regular  drift (one can see also \cite{NoDePe} for  some convergence rate in a negative Sobolev norm).

Here, we study   the incompressible limit of \eqref{pm} subject to Dirichlet boundary condition and compatible initial data (even changing sign  data). The reaction term satisfies general conditions, including Lipschitz continuous assumptions, and   the given  velocity field enjoys Sobloev regularity and an outpointing  condition   on the boundary that we'll precise after.     To this aim we use $L^1-$nonlinear  semi-group  theory, more or less, in the same spirit of the approach of Bénilan and Crandall \cite{BeCr3}. This consists in performing first the $L^1-$strong compactness   for the stationary problem and  work with the general theory of nonlinear semi group to pass to the limit in the evolution problem.      The $L^1-$compactness   enroll a new $BV_{loc}-$estimate we perform for the stationary problem in bounded domain with reasonable assumptions on $V$ in the neighborhood of the boundary.

%
%
%
%

 \subsection{Existence and uniqueness results}

  We assume that   $\Omega\subset\RR^N$ is a bounded  open set, with regular boundary   $\partial \Omega$ (say, piecewise  $\C^2$).  Throughout the paper, we assume that $V\in W^{1,2}(\Omega)$,  $\nabla \cdot V\in L^\infty(\Omega) $   and satisfies  the following  outward pointing condition on the boundary   :  
  \begin{equation}\label{HypV0} 
  	V\cdot \nu \geq  0\quad \hbox{ on }\partial \Omega,
  \end{equation}
where   $\nu$ represents the   outward unitary normal to  the boundary $\partial \Omega.$  Notice here that this condition   is fundamental in the case of Dirichlet boundary condition. This assumption is natural in many applications.  Even if it  looks alike to be stronger, it is fundamental for the uniqueness of weak solution for the limiting problem.  A counter example for uniqueness of weak solutions for a Hele-Shaw problem is given in \cite{Igshaw} whenever this condition is not fulfilled.
 
See here that $V\cdot \nu \in H^{-\frac{1}{2}}(\partial\Omega),$  so that    \eqref{HypV0}  needs to be understood a priori  in a weak sense ; i.e. 
 \begin{equation}\label{Vhun}
 	\int_\Omega V\cdot \nabla \xi \: dx +\int_\Omega \nabla \cdot V\: \xi\: dx \geq  0,\quad \hbox{ for any }0\leq \xi\in H^1(  \Omega).
 \end{equation}
To deal with this assumption,  we operate  technically with the euclidean distance-to-the-boundary function $d(.,\partial \Omega)$. For any $h>0,$  we denote by 
 \begin{equation}\label{xih}
 	\xi_h(t,x)=\frac{1}{h}\min\Big\{h,d (x,\partial \Omega)\Big\} \quad \hbox{ and } \quad \nu_h(x)=-\nabla \xi_h , \quad \hbox{ for any }x\in \Omega.
 \end{equation} 
 We see that    $    \xi_h \in H^1_0(\Omega) $, $0\leq \xi_h\leq 1$ in $\Omega$ and
 $$ \nu_h(x) =  - \frac{1}{h}\nabla\:  d(.,\partial \Omega) ,\quad \hbox{ for any }x\in \Omega\setminus \Omega_h=:D_h \hbox{ and }  0<h\leq h_0 \hbox{ (small enough)},$$ 
 where    \begin{equation}
 	\Omega_h=\Big\{ x\in \Omega\: :\: d(x,\partial \Omega)>h \Big\},\quad \hbox{ for small }h>0. 
 \end{equation} 
 In particular, for any $x\in \Omega_h,$ we have
 \begin{equation}\label{defnuh}
 	    \nu_h(x) = \frac{1}{h}\nu(\pi(x)), 
 \end{equation}  where $\pi(x)$  design the projection of $x$ on the boundary $\partial \Omega.$
    Thanks to \eqref{Vhun}, we see that   
 \begin{equation}\label{HypV1}
 	\liminf_{h\to 0}   \int_{\Omega\setminus \Omega_h}  \xi\:   V(x)   \cdot \nu_h(x)  \: dx \geq   0 , \quad \hbox{ for any }0\leq \xi \in H^1(\Omega).
 \end{equation}
  Nevertheless,  to avoid much more technicality in the proof of of  uniqueness, we'll assume that $V$ satisfies  \eqref{HypV1} for any $0\leq \xi\in L^2(\Omega)$  (cf. Remark \ref{RemboundaryCond2}). We do not know if this is a consequence of  the assumption \eqref{HypV0}. Anyway, this condition remains be to true for a large class of  practical situations and implies necessarily \eqref{HypV0}.

 \begin{remark}
  Remember that, thanks to the local $\C^2-$boundary regularity assumption on $\Omega,$  for any $0\leq \Phi\in H^1_0(\Omega),$   we have 
 		\begin{equation}\label{h10prop}
 			\liminf_{h\to 0}  	\int_\Omega \nabla \Phi\cdot \nabla \xi_h\: dx \leq 0 . 
 		\end{equation}
For  the case of  Lipschitz boundary domain, one needs  to work with   more sophisticated  test functions in the spirit of $\xi_h$ to fill \eqref{h10prop} like property (one can see Lemma 4.4 and Remark 6.5 of \cite{AnIg2}.  So, typical  examples of vector fields $V$ which satisfies \eqref{HypV1} may be given by 
 		\begin{equation}
 			V=-\nabla \Phi \quad \hbox{ and }\quad 0\leq \Phi\in H^1_0(\Omega)\cap W^{2,2}(\Omega). 
 		\end{equation}
 	Here       $H^1_0(\Omega)$ denotes the usual Sobolev space
 	$$H^1_0(\Omega) =\Big\{ u\in H^1(\Omega)\: :\  u=0, \: \:  \mathcal L^{N-1}\hbox{-a.e. in }\partial \Omega \Big\} . $$
 		 \end{remark}

 \bigskip 
 
 We consider the evolution problem 
  \begin{equation} 	\label{pmef}
  	\left\{  \begin{array}{ll} 
  		\displaystyle \frac{\partial u }{\partial t}  -\Delta u^m +\nabla \cdot (u  \: V)=f  \quad  & \hbox{ in } Q:= (0,T)\times \Omega\\  \\  
  		\displaystyle u= 0  & \hbox{ on }\Sigma := (0,T)\times \partial \Omega\\  \\   
  		\displaystyle  u (0)=u _0 &\hbox{ in }\Omega.
  	\end{array} \right.
  \end{equation}
 %

  \begin{definition}[Notion of solution] \label{defws} A  function $u  $  is said to be a weak solution of   \eqref{pmef}
  	if $u \in  L^2(Q)$, $p:=u^m\in     L^2
  	\left(0,T;H^1_0(\Omega)\right)$ and
  	\begin{equation}
  		\label{evolwf}
  		\displaystyle \frac{d}{dt}\int_\Omega u \:\xi+\int_\Omega ( \nabla p -  u \:V) \cdot  \nabla\xi   =     \int_\Omega f\: \xi  , \quad \hbox{ in }{\D}'(0,T),\quad \forall \: 	\xi\in H^1_0(\Omega).
  	\end{equation}
  	We'll say plainly  that $u$ is a solution of \eqref{pmef} if $u\in \C([0,T),L^1(\Omega))$, $u(0)=u_0$ and  $u$ is a weak solution of   \eqref{pmef}.
  \end{definition}

\medskip  
We denote by $\signp$  (resp. $\sign^-$)  the maximal monotone graph given by  
$$\signp(r)= \left\{ \begin{array}{ll}
	\displaystyle 1 &\hbox{ for  }r>0\\  
	\displaystyle [0,1]\quad & \hbox{ for } r=0\\
	\displaystyle 0 &\hbox{ for  }r<0. \end{array}\right. \quad \hbox{ (resp.  } \sign^-(r)=\sign^+(-r),\hbox{ for  }r\in \RR).$$ 
Moreover, we denote by       $\sign^{\pm}_0 $ the   discontinuous   applications defined from $\RR$ to $\RR$ by  
$$ \spo(r)= \left\{ \begin{array}{ll}
	\displaystyle 1 &\hbox{ for  }r>0\\   
	\displaystyle 0 &\hbox{ for  }r\leq 0 \end{array}\right.    \quad   \hbox{ and } \sign^-_0(r)= \sign_0^+(-r),\hbox{ for  }r\in \RR.  $$   

  \bigskip 
  
  As we said above, to avoid more technicality of the proofs of existence and uniqueness of a weak solution,  we assume throughout this section that $V$ satisfies the following outpointing condition on the boundary : 
   
  \begin{equation}\label{HypV}
  	\liminf_{h\to 0}  \frac{1}{h}  \int_{\Omega\setminus \Omega_h}  \xi\:   V(x)   \cdot \nu(\pi(x))  \: dx \geq   0 , \quad \hbox{ for any }0\leq \xi \in L^2(\Omega).
  \end{equation}

  \begin{theorem} \label{tcompcmef}
  	If $u_1$ and $u_2$ are two weak solutions of \eqref{pmef}  associated with $f_1,\ f_2\: \in L^1(Q)$ respectively, then  there exists $\kappa\in L^\infty(Q)$ such that $\kappa\in \signp(u_1-u_2)$ a.e. in $Q$ and  
  	\begin{equation}
  		\label{evolineqcomp}
  		\frac{d}{dt}	\int_\Omega ( u _1-u _2)^+ \: dx \leq \int_\Omega   \kappa \:  ( f_1-f_2)\:  dx ,\quad \hbox{ in }\D'(0,T).
  	\end{equation}
  	In particular, we have
  	\begin{enumerate}
  		\item $ 	\frac{d}{dt} \Vert u_1-u_2\Vert_{1} \leq \Vert f_1-f_2\Vert _{1},$  in $\D'(0,T).$
  	
  	\item If $f_1\leq f_2,$  a.e. in  $Q,$ and  $u_1(0)\leq u_2(0) $ a.e. in $\Omega,$   then
  	$$u_1\leq u_2,\quad \hbox{ a.e. in  }Q.$$
  	
\end{enumerate}
  \end{theorem}

  \begin{theorem}\label{texistevolm}
For any $u_0\in L^2(\Omega) $ and $f\in L^2(Q),$ the problem \eqref{pmef} has a   solution $u.$ 
  	Moreover, $u$ satisfies the following : 
  	\begin{enumerate}
  		
  		\item For any $q\in [1,\infty],$ we have  
  		\begin{equation} \label{lquevol} 
  			\Vert u(t)\Vert_q    \leq  M_q:=  \left\{ 
  			\begin{array}{lll}
  				e^{ (q-1)\:  T\:  \Vert (\nabla \cdot V)^-\Vert_\infty  }  \left( \Vert u_0 \Vert_q +   \int_0^T   \Vert f (t)\Vert_q \: dt  \right ) \quad &\hbox{ if } & q<\infty  \\ 
  				e^{ T \: \Vert (\nabla \cdot V)^-\Vert_\infty  }  \left( \Vert u_0 \Vert_\infty +   \int_0^T   \Vert f (t)\Vert_\infty \: dt  \right)&\hbox{ if } & q=\infty  
  			\end{array}\right.  . 
  		\end{equation}

  		\item  For any $t\in [0,T),$ we have 
  		\begin{equation}\label{lmuevol} 
  			\begin{array}{c} 
  				\frac{1}{m+1}  \frac{d}{dt} \int_\Omega \vert u\vert ^{m+1} \: dx+  \int_\Omega \vert \nabla p\vert^2  \: dx\leq     \int_\Omega f\:p \: dx +       \int     p \: u  \:  (\nabla \cdot  V )^-    \: dx, \quad \hbox{ in }\D'(0,T).
  			\end{array} 
  		\end{equation} 	       
  		
  	\end{enumerate}
  \end{theorem}

  \begin{corollary}\label{cpositif}
   For any $0\leq u_0\in L^2(\Omega) $ and $0\leq f\in L^2(Q),$ the problem \eqref{pmef} has a   unique nonnegative solution $u$.
  \end{corollary}

   \begin{remark}\label{RemboundaryCond1}
   	
   	The assumption \eqref{HypV}	 is technical for the the proof of Theorem \eqref{tcompcmef}. This assumption is fulfilled for a large  class of vector field $V$, like for instance the case where $V$ is  outward pointing in a neighborhood of  the boundary.  We think that this condition could be  removed if favor of merely \eqref{HypV0}  for instance if the solution $\rho$ has a trace on the boundary (for instance if one works with $BV$ solution).   We postpone the technicality of this assumption in  Remark \ref{RemboundaryCond1} after the proof of Theorem \ref{tcompcmef}.

   \end{remark}

 \subsection{Incompressible limit results}  
 
As we said in the introduction,   as $m\to\infty,$  the problem \eqref{pmef}   converges formally to  so called Hele-Shaw problem 
 \begin{equation} \label{evolhs0}
	\left\{  
	\begin{array}{ll}\left.
		\begin{array}{l}
			\displaystyle \frac{\partial u }{\partial t}  -\Delta p + \nabla \cdot (u\: V) =f   \\
			\displaystyle u \in \sign(p)\end{array}\right\}
		\quad  & \hbox{ in } Q \\  \\
		\displaystyle p= 0  & \hbox{ on }\Sigma  
	\end{array} \right.
\end{equation}
Existence, $L^1-$comparison  and uniqueness  of weak solution for the problem \eqref{evolhs0}, with mixed boundary conditions,  has been studied recently in \cite{Igshuniq} (see also \cite{IgshuniqR}) under the assumption  \eqref{HypV}. Thanks to \cite{Igshuniq}, we know that  for any $f\in L^2(Q)$ and $u_0\in L^\infty(\Omega),$ s.t. $0\leq \vert u_0\vert \leq 1,$ a.e. in $\Omega,$  \eqref{evolhs0} has a unique weak solution (see the following Theorem for the precise sense) satisfying $u(0)=u_0.$

\medskip 
 To prove rigorously the convergence of $u_m$ to the solution of   \eqref{evolhs0}, we assume moreover  that $V$ satisfies the following assumption :  there exists  $h_0>0,$ such that   for any $0< h< h_0,$ we have  
 
\begin{equation}   	\label{HypsupportV}
V(x)\cdot \nu(\pi(x)) \geq 0,\quad \hbox{ for a.e. } x\in D_h.
\end{equation} 
  This  may be written also as  $V(x)\cdot \nabla d(x,\partial\Omega) \leq 0,$ for any $x\in \Omega $   being such that $d(x,\partial\Omega) < h_0.$ See that the condition \eqref{HypsupportV} implies definitely  \eqref{HypV}.     In fact, with \eqref{HypsupportV}, we are assuming  that  $V$ is outpointing along the paths given by the distance function in a neighborhood of $\partial \Omega.$  As we will see, this assumption can be weaken into an outpointing vector field condition  along a given arbitrary paths in the neighborhood of $\partial \Omega$ (cf. Remark \ref{Rbvcond}).     Nevertheless, a control of the outpointing orientation of $V$ in the neighborhood of the boundary  seems to be important in order to handle the oscillation of $u_m$ and establish $BV_{loc}$-estimate.

\begin{theorem}\label{treglimum}
Under the assumption \eqref{HypsupportV}, 	for each $m=1,2,...,$ we consider  $u_{0m}\in L^2(\Omega),$ $f_m\in L^2(Q)$ and $u_m$  be the corresponding     solution of \eqref{pmef}. 
	If,   as $m\to\infty,$    $	f_{ m} \to f $ in $L^1(Q),$ $u_{0m} \to u_0 $ in $L^1(\Omega), $	 	and   $\vert u_{0}\vert \leq 1,$ then   
	\begin{equation}
		u_m\to u,\quad \hbox{ in }\C([0,T);L^1(\Omega)), 
	\end{equation}
	\begin{equation}
		u_m^m\to p,\quad \hbox{ in }L^2([0,T);H^1_0(\Omega))\hbox{-weak},
	\end{equation}
	and $(u,p)$ is the unique solution of \eqref{evolhs0} satisfying $u(0)=u_0.$ That is  $ u  \in \C([0,T),L^1(\Omega))  ,$ $u(0)=u_0$,     $u\in \sign(p) $ a.e. in $Q,$ and  
	\begin{equation}\label{weakformhs}
		\frac{d}{dt} \int_\Omega u\: \xi+ \int_\Omega \nabla p\cdot \nabla \xi\: dx -\int_\Omega u\: V\cdot \nabla \xi\: dx =\int_\Omega\: f\: \xi\: dx,\quad \hbox{ in }\D'([0,T)),  \hbox{ for any }\xi\in H^1_0(\Omega).
	\end{equation}  
\end{theorem} 
The results for the case where $f$ is given by a reaction term $g(.,u),$ are develop in  Section \ref{Sreaction}.

\begin{remark}
Thanks to \cite{Igshuniq}, we can deduce that $u,$ the limit of $u_m,$ satisfies the following  : 
\begin{enumerate}
	\item   If there exists $\omega_1,   \in W^{1,1}(0,T)$  (resp.  $  \omega_2  \in W^{1,1}(0,T)$) such that  $ u_0\leq \omega_2(0)$ (resp. $\omega_1(0)\leq u_0 $) and, for any $t\in (0,T),$  
	\begin{equation}\label{sup}
		\dot \omega_2(t)+ \omega_2(t)\nabla \cdot V \geq  f(t,.) \quad \hbox{ a.e. in  }\Omega
	\end{equation}
 (rep. $\dot \omega_1(t)+ \omega_1(t)\nabla \cdot V \leq   f(t,.),$ a.e. in  $\Omega), $ 
 then   we have  
	\begin{equation}
		u\leq \omega_2 \quad (\hbox{resp. }\omega_1\leq u)\quad \hbox{  a.e. in }Q. 
	\end{equation}  
	
	\item   If $f$ and $V$ satisfies   
	\begin{equation}
	    0\leq f \leq \nabla \cdot V    ,  \hbox{ a.e.  in } Q   
	\end{equation}
	then $p\equiv 0,$ and $u$ is     the unique solution of the reaction-transport equation 
	\begin{equation}
		\label{cmep0}
		\left\{  \begin{array}{ll}
			\left. \begin{array}{l}
				\displaystyle \frac{\partial u }{\partial t}   +\nabla \cdot (u  \: V)= f\\ 
			0\leq  u  \leq 1 
			\end{array}\right\} 	\quad  & \hbox{ in } Q \\  
			\displaystyle u  \: V\cdot \nu = 0  & \hbox{ on }\Sigma_N \\   
			\displaystyle  u (0)=u _0 &\hbox{ in }\Omega,\end{array} \right.
	\end{equation}
	in the sense that $u  \in \C([0,T),L^1(\Omega)),$  $0\leq  u  \leq 1$  a.e. in $Q$  and
	\begin{equation}
		\label{evolwg}
		\displaystyle \frac{d}{dt}\int_\Omega u \:\xi- \int_\Omega   u \:V \cdot  \nabla\xi   =     \int_\Omega  f\: \xi  , \quad \hbox{ in }{\D}'(0,T), \quad \forall \: \xi\in H^1_0(\Omega).
	\end{equation}
\end{enumerate}

\end{remark}

\begin{remark}\label{remV}
	
	One sees that the assumption \eqref{HypsupportV} is fulfilled for instance in the following cases  : 
	 
		\begin{enumerate}
		
		\item $V$ satisfies \eqref{HypV0},  and there exists $h_0>0$ such that, for any $0<h<h_0,$ we have 
		\begin{equation}
			V(x)=V(\pi(x)),\quad \hbox{ for any }x\in \Omega_h.
		\end{equation} 
		Indeed, since $\nu_h(x)= \nu(\pi(x))$, we have  $ V(x)\cdot \nabla \xi_h(x)= V(\pi(x))\cdot \nabla \xi(\pi(x)) $ which is nonnegative by the assumption \eqref{HypV0}.

			\item $V$ compactly supported  ; i.e.  $V$ vanishes on a neighbor of the boundary $\partial \Omega.$ 
			
		\end{enumerate}

\end{remark}

\begin{lemma}
	 Under the assumption \eqref{HypsupportV},   there exists  $h_0>0,$ such that 
	for any $0<h<h_0,$ there exists $0\leq \omega_h\in \mathcal C^2(\Omega_h)$ compactly supported in $\Omega,$ such that $\omega_h \equiv 1$ in  $    \Omega_h$ and 
	\begin{equation}  \label{HypsupportV2}  
		\	\int_{\Omega\setminus \Omega_h}  \varphi\:  V\cdot \nabla \omega_h  \: dx \leq 0,\quad \hbox{ for any }0\leq \varphi\in L^1(\Omega).
	\end{equation}
	
	\end{lemma}

 \subsection{Plan of the paper}  
 The next section is devoted to  the proof of $L^1-$comparison principle for   weak solutions of \eqref{pmef}. To this aim, we use doubling and dedoubling variables techniques. This enables us to deduce the uniqueness and lay out the study plan of the equation in the framework of  $L^1-$nonlinear semi-group  theory.  
 Section 3 concerns the study  of  existence of a solution. To set the problem in the framework of nonlinear semi group theory, we   begin with stationary problem to operate the  Euler-implicit discretization and construct  an $\eps-$approximate solution $u_\eps.$ Then, using mainly  a Crandall-Ligget type theorem, $L^2(\Omega)$ and  $H^1_0(\Omega)$  estimates on  $u_\eps$ and $u_\eps^m$  respectively, we pass to the limit as $\eps\to 0,$ to built $u_m$ the solution of the evolution problem \eqref{pmef}. Section 4 is devoted to the study of the limit as $m\to\infty.$ Using the outpointing  vector filed condition   \eqref{HypsupportV}, we study first the limit  for the stationary problem  connecting it to the   the Hele Shaw flow with linear drift. To this aim, we  establish  $BV_{loc} $ new estimates for weak solutions in bounded domain. Then, using regular perturbation results for nonlinear semi group we establish the convergence results for the evolution problem.  Section 6  is devoted to the study of the limit of the solution $u$ and $u^m$ in the  of  the presence of a reaction term  with linear drift. We prove the convergence  of reaction diffusion problem of a Hele-Shaw flow with linear drift
 At last, in Section 7 (Appendix),  we provide for the unaccustomed  reader a  short recap on the main tools from $L^1-$nonlinear semi-group theory.

\section{$L^1-$comparison principle and uniqueness proofs}
 \setcounter{equation}{0}
 
As usual for parabolic-hyperbolic and  elliptic-hyperbolic problems, the main tool to prove the uniqueness is doubling and de-doubling variables. To this aim, we prove first that a weak solution satisfies the following  version of entropic inequality :

\medskip 
We assume throughout this section that $V\in W^{1,2}(\Omega)$ ,  $(\nabla \cdot V)\in L^\infty(\Omega) $   and $V$ satisfies the outpointing  condition \eqref{HypV0}. 
\begin{proposition}\label{pentropic}
	Let $f\in L^1(Q)$ and  $u$ be a weak solution of \eqref{pmef}.  Then, for any $k\in \RR$,  and $0\leq \xi\in H^1_0(\Omega)\cap L^\infty(\Omega),$  we have    
	\begin{equation}\label{entropic+}
		\begin{array}{c}
		\frac{d}{dt}\int_\Omega  (u-k)^+ \xi\: dx 	+ \int_\Omega  (\nabla ({u^m}-k^m)^+  -   (u-k)^+  V)  \cdot \nabla \xi \: dx   \\ 
		  +\int_\Omega ( k   \:  \nabla \cdot V -f )  \: \xi\:  \spo(u-k) \: dx      \leq   -\limsup_{\eps\to 0 }\frac{1}{\eps}\int_{[0\leq u^m-k^m \leq \eps]} \vert \nabla u^m\vert^2\: \xi\: dx ,	\end{array}
	\end{equation}  
and   
	\begin{equation} \label{entropic-} 
			\begin{array}{c}
			\frac{d}{dt}\int_\Omega  (k-u)^+ \xi\: dx 	+ \int_\Omega ( \nabla (k^m-{u^m})^+  -   (k-u)^+  V)  \cdot \nabla \xi \: dx   \\ 
			+\int_\Omega (f - k   \:  \nabla \cdot V )  \: \xi\:  \spo(k-u) \: dx      \leq  -\limsup_{\eps\to 0 }\frac{1}{\eps}\int_{[0\leq k^m-u^m \leq \eps]} \vert \nabla u^m\vert^2\: \xi\: dx, 	\end{array}
	\end{equation} 	
 in  $ \D'(0,T)  .$	\end{proposition}
\begin{proof}   We extend $u$ onto $\displaystyle \RR\times\Omega$ by
	$0$ for any   $\displaystyle t\not\in (0,T).$   Then, for any $h>0$ and nonnegative $\xi \in H^1_0(\Omega)$ and $\psi\in \D(0,T),$  we  consider
	$$\displaystyle \Phi^h (t,x)=    \xi(x)\:    \frac{1}{h}\int_t^{t+h}  \hepsp (u^m(s,x)\:   \psi(s)\: ds  ,\quad \hbox{ for a.e. }x\in \Omega, $$
	where  we extend $\psi$ onto $\RR$ by $0$, and    $\hepsp$ is given by
	\begin{equation} \label{hepsp}
		\hepsp(r)=  \min \left( \frac{(r-k^m)^+}{\eps}, 1 \right),\quad \hbox{ for any }r\in \RR,
	\end{equation}
	for arbitrary  $\sigma >0.$  It is clear that  $\Phi_h \in W^{1,2} \Big( 0,T;H^1_0(\Omega) \Big)\cap L^\infty(Q)$    is an admissible test function for the weak formulation, so that 
	\begin{equation}\label{evolh0}
		-	\int\!\!\!\int _Q   u  \: \partial_t   \Phi^h \: dtdx	 +	\int\!\!\!\int _Q (\nabla  u^m - V\:  u   ) \cdot \nabla \Phi^h\: dtdx \\  \\ =  	\int\!\!\!\int _Q f\:  \Phi^h\: dtdx.
	\end{equation}
	See that 
	\begin{equation} \label{rel0}
		\begin{array}{ll} 
			\int\!\!\!\int _Q   u  \: \partial_t   \Phi^h   \: dtdx&=  	\int\!\!\!\int _Q
	\psi(t)\:  	\heps^+(u^m(t))\:	 \frac{ u(t-h)-u(t))}{h} \:   \xi \: dtdx\\ 
	&\leq  \frac{1}{h} 	\int\!\!\!\int _Q
	\psi(t)\:  	 	 \left(   \int_{u(t)}^{u(t-h)}\hepsp(r^m)\: d r\right) \:   \xi\ \: dtdx \\ 
	&\leq   \frac{1}{h} 	\int\!\!\!\int _Q \left( \int_{k}^{u(t)}	\hepsp 	(r^m)dr  \right)  \:  	  ( \psi(t+h ) -\psi(t)   )  \xi\: \: dtdx .
	\end{array}
	\end{equation}
Letting $h\to 0,$ we have 
\begin{equation}
\limsup_{h\to 0}	\int\!\!\!\int _Q   u  \: \partial_t   \Phi^h \: dtdx\leq 	\int\!\!\!\int _Q
\left( \int_{k}^{u(t)}	\hepsp 	(r^m)dr  \right ) \:  	\partial_t\psi \: \xi\: dtdx.
\end{equation} 
 So,   by  letting $h\to 0$ in \eqref{evolh0}, we get   
 	\begin{equation}\label{inegproof1}
		\begin{array}{c} 
		-	\int\!\!\!\int _Q
	\left\{ 	\Big( \int_{k}^{u(t)}	\hepsp	(r^m)dr  \Big ) \:  	\partial_t\psi \: \xi 	+     \psi\:  \nabla u^m \cdot  \nabla\xi \hepsp(u^m)\xi - 
		  \hepsp(u^m)  (u-k)  \:  V\cdot \nabla  \xi \right\} \: dtdx \\ 
			\leq       \int\!\!\!\int _Q	\Big\{   \psi\: (f  +k   \:  \nabla \cdot V)\: \hepsp(u^m)\xi  
			+       \psi\: \xi    (u-k)  \:  V\cdot \nabla  \hepsp(u^m) \Big\} \: dtdx   \\
			- \frac{1}{\sigma}\int\!\!\!\int _{[0\leq u^m-k^m\leq \sigma]} \vert \nabla u^m\vert^2 \: \xi \: dtdx  ,
	\end{array} 	\end{equation}
where we use the fact that $  \nabla u^m \cdot\nabla \hepsp(u^m))= \frac{1}{\sigma}\vert \nabla u^m\vert^2 \:  \chi_{[0\leq u^m-k^m\leq \sigma]}    $ a.e. in $Q.$  Setting  
$$\Psi_\sigma :=   \frac{1}{\sigma}    \int_{\min(u^m,k^m)}^{\min(u^m,k^m+\sigma)}  (r^{1/m}-k)   \: dr,$$
we see that  
$$ (u-k)  \heps'(u^m-k^m)  \cdot \nabla u^m= \nabla \Psi_\sigma.$$ 
  This implies that the last term of \eqref{inegproof1}    
satisfies 	\begin{eqnarray}
	   \int\!\!\!\int _Q \psi\: \xi    (u-k)  \:  V\cdot \nabla  \hepsp(u^m-k^m)  &=& 	\int\!\!\!\int _Q  \xi    (u-k)  {\hepsp}'(u^m-k^m)   \:  V\cdot \nabla u^m    \\ &=&  	  \int\!\!\!\int _Q \psi\:\xi \:   V\cdot \nabla \Psi_\sigma   dx 
		\\ &=&    -     \int\!\!\!\int _Q \psi\:   \nabla\cdot(  \xi \:   V) \: \Psi_\sigma \:  dx   \\   \\ 
		&  \to& 0,\quad \hbox{ as }\sigma\to 0.
	\end{eqnarray} 
See also that, by using Lebesgue’s dominated convergence Theorem, we have 
\begin{equation}
	\limsup_{\eps\to 0}	\int\!\!\!\int _Q
	\left( \int_{k}^{u(t)}	\hepsp 	(r^m)dr  \right ) \:  	\partial_t\psi \: \xi= 	\int\!\!\!\int _Q
	(u(t) -k)^+  \:  	\partial_t\psi   \:   \xi. 
\end{equation}
Then,  letting $\eps \to 0$ in \eqref{inegproof1} and using the fact that $\sign_0^+(u^m-k^m)= \sign_0^+(u-k)$, for any $k\in \RR,$ we get \eqref{entropic+}.  As to  \eqref{entropic-}, it follows by  using the fact that   $-u$ is also a solution 
	of  \eqref{pmef} with $f$ replaced by $-f,$  and   applying    \eqref{entropic+} to $-u$.  
\end{proof}

\bigskip
\begin{proposition}[Kato's inequality]\label{PKato}
	If   $u_1$ and $u_2$  satisfy   \eqref{entropic+}  and  \eqref{entropic-} corresponding to $f_1\in L^1(Q)$ and  $f_2\in L^1(Q)$ respectively, then    
	\begin{equation}\label{ineqkato}
		\begin{array}{c}
			\partial_ t	( u_1-u_2 )^+
			- \Delta    (u_1^m-u_2^m  )^+ + \nabla \cdot \left(
			( u_1-u_2 )^+  \: V \right)      \leq     ( f_1-f_2 )\:  \sign^+_0(u_1-u_2) \,
			\hbox{ in }\D'(Q).
		\end{array}
	\end{equation}
\end{proposition}
\begin{proof}
	The proof of this lemma is based on  doubling and de-doubling variable techniques.  Let us give here briefly the arguments.  To double the variables,  we use first the fact that $ u_1=u_1(t,x) $ satisfies  \eqref{entropic+}  with $k=u_1(s,y), $ we have
	\begin{equation}
		\begin{array}{c}  \frac{d}{dt}   \int     (u_1(t,x)-u_2(s,y)  )^+   \:
			\zeta \: dx + \int   (\nabla_x   (u_1^m (t,x)-u_2^m(s,y))^+ - ( u_1(t,x)-u_2(s,y)  )^+ \:   V(x)   \cdot   \nabla_x\zeta  \: dx     \\      + \int_\Omega  \nabla_x \cdot V \:  u_2(s,y)  \zeta   \spo   (u_1(t,x)-u_2(s,y)   ) \: dx   \leq \int f_1(t,x) \spo   (u_1(t,x)-u_2(s,y)   )  \: \zeta\: dx  \\  
			-\limsup_{\eps\to 0 }\frac{1}{\eps}\int_{[0\leq u_1^m-u_2^m \leq \eps]} \vert  \nabla_x u_1^m(t,x)\vert^2\: \zeta\: dx  . \end{array}
	\end{equation}
Integrating with respect to $y,$ we get 
	\begin{equation}
	\begin{array}{c}  \frac{d}{dt}   	 \int\!\! \int       (u_1(t,x)-u_2(s,y)  )^+   \:
		\zeta  + 	 \int\!\! \int     (\nabla_x   (u_1^m (t,x)-u_2^m(s,y))^+ - ( u_1(t,x)-u_2(s,y)  )^+ \:   V(x)   \cdot   \nabla_x\zeta     \\      + 	 \int\!\! \int    \nabla_x \cdot V \:  u_2(s,y)  \zeta   \spo   (u_1(t,x)-u_2(s,y)   )  \leq 	 \int\!\! \int  f_1(t,x) \spo   (u_1(t,x)-u_2(s,y)   )  \: \zeta \\  
		-\limsup_{\eps\to 0 }\frac{1}{\eps} \int\!\! \int  _{[0\leq u_1^m-u_2^m \leq \eps]} \vert \nabla_x u_1^m(t,x)\vert^2\: \zeta  . \end{array}
\end{equation}
 See that   
 \begin{equation}
 	\begin{array}{c}
 	 \int\!\! \int     \nabla_y   (u_1^m (t,x)-u_2^m(s,y))^+   \cdot   \nabla_x\zeta  \: dxdy = - \lim_{\eps\to 0}  	 \int\!\! \int     \nabla_y  u_2^m(s,y)   \cdot   \nabla_x\zeta  \: H_\sigma (u_1^m (t,x)-u_2^m(s,y)) \: dxdy \\  \\ 
 = - \lim_{\eps\to 0}  	\frac{1}{\eps} \int\!\! \int  _{[0\leq u_1^m-u_2^m \leq \eps]}    \nabla_x  u_1^m (t,x) \cdot \nabla_y  u_2^m(s,y)  \: \zeta  \:   dxdy  ,
 	 	\end{array}
  	 \end{equation}	
  	 so that, denoting  by
	$$u(t,s,x,y)=  u_1(t,x) -u_2(s,y)   ,\quad \hbox{ and }  \quad    p(t,s,x,y)=  u_1^m (t,x)-u_2^m(s,y)   ,$$
	 we obtain
	\begin{equation}\label{eq00}
		\begin{array}{c}  \frac{d}{dt}   \int\!\! \int     u(t,s,x,y)^+   \:
			\zeta  \: dxdy +   \int\!\! \int\Big\{   (\nabla_x+\nabla_y)  p(t,s,x,y)   -     u(t,s,x,y) ^+  \:   V(x)\Big\}    \cdot   \nabla_x\zeta     \: dxdy   \\    +  \int\!\! \int    \nabla_x \cdot V \:  u_2(s,y) \:  \zeta\:    \spo   u(t,s,x,y)  \: dxdy    \leq     \int\!\! \int   f_1(t,x) \spo   u(t,s,x,y)    \: \zeta  \: dxdy  \\ 
			- \lim_{\eps\to 0}   \int\!\!	\int     \nabla_x  u_1^m (t,x) \cdot \nabla_y  u_2^m(s,y)  \: \zeta  \: H'_\eps(u_1^m (t,x)-u_2^m(s,y)) \: dxdy    \\  
			-\limsup_{\eps\to 0 }\frac{1}{\eps} \int\!\! \int  _{[0\leq u_1^m-u_2^m \leq \eps]} \vert \nabla_x u_1^m(t,x)\vert^2\: \zeta  \: dxdy      . \end{array}
	\end{equation}
	On the other hand,   using the fact that $ u_2=u_2(s,y)$ satisfies  \eqref{entropic-} with $k=u_1(t,x),$ we have
	\begin{equation}
		\begin{array}{c}
			\frac{d}{ds}   \int    u(t,s,x,y)^+   \:
			\zeta   \: dy   + 	  \int   \nabla_y   p(t,s,x,y) - u(t,s,x,y) ^+ \:   V(y)   \cdot   \nabla_y\zeta      \\    - \int_\Omega  \nabla_y \cdot V \:  u_1(t,x)  \zeta   \spo  (u(t,s,x,y))  \: dy   \leq - \int f_2(s,y) \spo ( u(t,s,x,y) ) \: \zeta  \: dy   \\ 	-\limsup_{\eps\to 0 }\frac{1}{\eps} \int  _{[0\leq u_1^m-u_2^m \leq \eps]} \Vert \nabla_y u_2^m(s,y)\Vert^2\: \zeta \: dy    .\end{array}
	\end{equation}
	Working in the same way for \eqref{eq00},    we get
	\begin{equation}
		\begin{array}{c}   
			\frac{d}{ds}   \int\!\! \int     u(t,s,x,y)^+  \:
			\zeta     \: dxdy    +    \int\!\! \int \Big\{   (\nabla_x+\nabla_y)  p(t,s,x,y)    -   u(t,s,x,y) ^+  \:   V(y)  \Big\}  \cdot   \nabla_y\zeta   \: dxdy \\     -   \int\!\!  \int   \nabla_y \cdot V(y) \:  u_1(t,x)   \:  \zeta\:    \spo (  u(t,s,x,y) )   \: dxdy  \leq   -  \int\!\! \int   f_2(s,y)  \spo   (u(t,s,x,y) )  \: \zeta    \: dxdy   \\ 
			- \lim_{\eps\to 0}  \frac{1}{\eps} \int\!\! \int  _{[0\leq u_1^m-u_2^m \leq \eps]}   \nabla_x  u_1^m (t,x) \cdot \nabla_y  u_2^m(s,y)  \: \zeta  \: dxdy    \\  -\limsup_{\eps\to 0 }\frac{1}{\eps} \int\!\! \int  _{[0\leq u_1^m-u_2^m \leq \eps]} \vert \nabla_y u_2^m(s,y)\vert^2\: \zeta \: dy \: dxdy    . \end{array}
	\end{equation}
	Adding both inequalities, and using the fact that 
	\begin{equation}
			-(\Vert \nabla_x u_1(t,x)\Vert ^2 +\Vert \nabla_y u_2(s,y)\Vert ^2 	+2\:  \nabla_x u_1 (t,x) \cdot \nabla_y u_2(x,y))\: \chi_{[0\leq u_1^m-u_2^m \leq \eps]} \leq 0 ,   \hbox{ a.e. in } Q^2,
	\end{equation}
	we obtain
	\begin{equation}\label{formdoubling1}
		\begin{array}{c}   \left(   \frac{d}{dt}  +   \frac{d}{ds}   \right)    \int\!\! \int u(t,s,x,y) ^+   \:
			\zeta   \: dxdy+ \int\!\! \int   (\nabla_x+\nabla_y)   p(t,s,x,y)    \cdot   (\nabla_x +\nabla_y)\zeta \: dxdy  \\ -  \int\!\! \int    u(t,s,x,y) ^+  \: (   V(x)   \cdot    \nabla_x\zeta +V(y)   \cdot  \nabla_y\zeta)    \: dxdy  \\  +  \int\!\!  \int  \left(   \nabla_x \cdot V(x) \:  u_2(s,y)    \: dxdy  -   \nabla_y \cdot V(y) \:  u_1(t,x)    \right) \:  \zeta\:    \spo   (u(t,s,x,y) )   \: dxdy \\  
			\leq   \int\!\!  \int  ( f_1(t,x) -  f_2(s,y) )  \spo ( u(t,s,x,y) )   \: \zeta   \: dxdy   \end{array}
	\end{equation}
and then, 
	\begin{equation}\label{formdoubling1}
	\begin{array}{c}   \left(   \frac{d}{dt}  +   \frac{d}{ds}   \right)    \int\!\! \int u(t,s,x,y) ^+   \:
		\zeta  \: dxdy  + \int\!\! \int   (\nabla_x+\nabla_y)   p(t,s,x,y)    \cdot   (\nabla_x +\nabla_y)\zeta  \: dxdy  \\ -  \int\!\! \int    u(t,s,x,y) ^+  \:    V(x)   \cdot   ( \nabla_x\zeta  + \nabla_y\zeta  )   \: dxdy  +    \int\!\! \int    u(t,s,x,y) ^+  \:  (  V(x)-V(y)    \cdot   \nabla_y\zeta  \: dxdy  \\  +  \int\!\!  \int  \left(   \nabla_x \cdot V(x) \:  u_2(s,y)    -   \nabla_y \cdot V(y) \:  u_1(t,x)    \right) \:  \zeta\:    \spo   (u(t,s,x,y) )  \: dxdy  \\  
		\leq   \int\!\!  \int  ( f_1(t,x) -  f_2(s,y) )  \spo ( u(t,s,x,y) )   \: \zeta \: dxdy .    \end{array}
\end{equation} 
	Now, we can de-double the variables $t$ and $s,$ as well as $x$ and $y,$ by taking as usual the sequence of test functions  
	\begin{equation}    
		\psi_\eps(t,s) = \psi\left( \frac{t+s}{2}\right)   \rho_\eps \left( \frac{t-s}{2}\right)    \hbox{ and }  
		\zeta_\lambda  (x,y) = \xi\left( \frac{x+y}{2}\right) \delta_\lambda \left( \frac{x-y}{2}\right),
	\end{equation}
	for any $t,s\in   (0,T)$ and $x,y\in \Omega.$  Here  $\psi\in \D(0,T),$  $\xi\in \D(\Omega),$      $\rho_\eps$ and   $\delta_\lambda$ are sequences of usual  mollifiers in $\RR$ and $\RR^N$ respectively.   
	See that 
	$$   \left(   \frac{d}{dt}  +   \frac{d}{ds}   \right)   \psi_\eps(t,s)  =  \rho_\eps\left( \frac{t-s}{2}\right)  \dot   \psi \left( \frac{t+s}{2}\right)   $$ 
	and 
	$$(\nabla_x+ \nabla_y)  \zeta_\lambda  (x,y)  =   \delta_\lambda \left( \frac{x-y}{2}\right)  \nabla \xi\left( \frac{x+y}{2}\right)  $$
	Moreover,   for  any $h\in L^1((0,T)^2\times \Omega^2)$ and   $\Phi\in L^1((0,T)^2\times  \Omega^2)^N,$   we know that 
	\begin{itemize}
		\item  $\lim_{\lambda \to 0 }\lim_{\eps \to 0 } \int_0^T\!\! \int_0^T\!\!   \int_\Omega\!\!  \int_\Omega h(t,s,x,y)\:\zeta_\lambda(x,y)\: \rho_\eps(t,s)   = \int_0^T\!\!  \int_\Omega  h(t,t,x,x)\: \xi(x)\:  \psi(t) .$
		
		\item  $\lim_{\lambda \to 0 }\lim_{\eps \to 0 } \int_0^T\!\! \int_0^T\!\!   \int_\Omega\!\!  \int_\Omega h(t,s,x,y)\:\zeta_\lambda(x,y)\: \left(\frac{d}{dt} + \frac{d}{ds}\right)\rho_\eps(t,s)  = \int_0^T\!\!   \int_\Omega  h(t,t,x,x)\: \xi(x)\:  \dot \psi(t) .$ 
		
		\item $\lim_{\lambda \to 0 }\lim_{\eps \to 0 }   \int_0^T\!\!  \int_\Omega\!\!  \int_\Omega \Phi(t,s,x,y) \cdot  (\nabla_x + \nabla_y) \zeta_\lambda(x,y) \: \rho_\eps(t,s)  = \int_0^T\!\!\int_\Omega  \Phi(t,t,x,x)\cdot \nabla \xi(x)\: \psi(t)\: dtdx   .$
	\end{itemize}  
Moreover, we also know that
\begin{equation}
	\lim_{\lambda \to 0 } \int_0^T\!\!  \int_\Omega\!\!  \int_\Omega    h(t,x,y) \:    (V(x) -V(y))   \cdot \nabla_y  \zeta_\lambda(x,y)  \:   \psi(t)\: dtdxdy
\end{equation}  
\begin{equation} \label{inqtech1}
	=  	    \int_0^T\!\!\int_\Omega   h(x,x)   \:        \nabla \cdot V(x) \: \xi(x)\:  \psi(t)\:  dtdx .
\end{equation}  
The proof of this result  is more or less well known by now (one can see for instance a detailed proof in \cite{Igshuniq}). 
	So replacing $\zeta$ and $\psi$ in \eqref{formdoubling1} by $\zeta_\lambda$ and $\psi_\eps$ resp.,  and,  letting $\eps,\: \lambda \to 0$,  we get  
	\begin{equation}\label{de-double }
		\begin{array}{c}
		 \int_0^T\!\! \int_\Omega \Big\{ - (u_1-u_2  )^+   \:
			\xi \: \dot \psi   +  \nabla (p_1-p_2)^+  \cdot \nabla \xi\:   \psi  
			-      (u_1-u_2  )^+  \:    V    \cdot   \nabla \xi  \:    \psi  \Big\} \: dtdx  \\   \leq   \int_0^T\!\! \int_\Omega  (f_1-f_2)\: \sign^+_0(u_1-u_2) \: \xi   \:   \psi   \: dtdx  . \end{array}
	\end{equation}
  Thus 
	\begin{equation}\label{dedouble}
		\begin{array}{c}
			\frac{d}{dt}    \int (u_1-u_2  )^+   \:
			\xi \: dx   + \int  \nabla   (u_1^m (t,x)-u_2^m(t,x))^+  \cdot \nabla \xi \: dx 
			-   \int  (u_1-u_2  )^+  \:   V    \cdot   \nabla \xi  \: dx  \\   \leq   \int \kappa (x) (f_1-f_2) \: \xi   \: dx  . 
		\end{array}
	\end{equation} 
Thus  the result of the proposition. 
\end{proof}

\bigskip 
The aim now  is to  process  with the sequence of test function $\xi_h$ given by \eqref{xih} in Kato's inequality and let $h\to 0,$ to cover  \eqref{evolineqcomp}.

 \bigskip
 \begin{proof}[\textbf{Proof of Theorem \ref{tcompcmef}}]  Let  $(u_1,p_1)$  and  $(u_2,p_2)$   be two  couples  of  $L^\infty(Q) \times  L^2	\left(0,T;H^1_0(\Omega)\right) $   satisfying  \eqref{entropic+}  and  \eqref{entropic-} corresponding to $f_1\in L^1(Q)$ and  $f_2\in L^1(Q)$ respectively, to prove     \eqref{evolineqcomp} we see  that
	\begin{eqnarray}
	&	\frac{d}{dt}	\int_\Omega 	( u_1-u_2 )^+ \: dx - \int   (f_1-f_2) \: \sign_0^+(u_1-u_2)\: dx   \\  & =   \lim_{h\to 0} \: \:  \underbrace{  \frac{d}{dt}  	\int_\Omega
		( u_1-u_2 )^+    \, \xi_h \: dx    -\int     (f_1-f_2)\: \sign_0^+(u_1-u_2)\:    \xi_h  \: dx 
	}_{I(h) }  ,
	\end{eqnarray}
	in the sense of distribution in $[0,T).$ Taking  $\xi_h$ as a test function in \eqref{ineqkato} and using  \eqref{h10prop},  we have
	\begin{equation}\label{inqxih}
\begin{array}{ll} 
			I(h)
	 & \leq  -    \int \left(  \nabla  ( u_1^m -u_2^m)^+   -
		( u_1-u_2 )^+  \: V \right)  \cdot\nabla  \xi_h  \: dx  \\ 
		& \leq   \int 
		( u_1-u_2 )^+  \: V \cdot\nabla  \xi_h  \: dx  . 
		\end{array}
	\end{equation}	
Then, using   the outpointing   velocity vector field  assumption \eqref{HypV}, we get  
	\begin{equation}\label{limintegral}
	\begin{array}{ll}
			\lim_{h\to 0}I(h)  	&\leq 	-\lim_{h\to 0}   \int
		( u_1-u_2 )^+  \: V   \cdot \nu_h(x) \: dx\\  \\
		&\leq  0.
	\end{array}
	\end{equation}
   Thus \eqref{evolineqcomp}. The rest of the theorem  is a straightforward  consequence of  \eqref{evolineqcomp}. 
\end{proof}

\begin{remark}\label{RemboundaryCond2}
	See that   \eqref{limintegral}  is the lonely step of the proof of Theorem \ref{tcompcmef} where we use assumption \eqref{HypV}.  Working so   enables to avoid all the technicality related to 
	doubling and de-doubling variable à la Carillo (\cite{Ca}) by using test functions which do not vanish on the boundary.  
	We do believe that the result of Theorem \ref{tcompcmef} remains to be true   under the general assumption \eqref{HypV0}.      One sees also that, if the solutions have a trace (like for $BV$ solution), one can  weaken this condition by handling \eqref{limintegral} otherwise.  \end{remark}

\section{Main estimates and existence proofs}
\setcounter{equation}{0}

  \subsection{Stationary problem} \label{Sapprox}
  
  To prove Theorem \ref{texistevolm},  we consider the   stationary problem associated with Euler-implicit discretization of  \eqref{pmef}. That is  
  \begin{equation}	\label{st}
  	\left\{  \begin{array}{ll} 
  		\displaystyle v -\lambda\Delta v^m + \lambda   \nabla \cdot (v  \: V)=f  
  		\quad  & \hbox{ in }  \Omega\\   
  		\displaystyle v= 0  & \hbox{ on }  \partial \Omega,\end{array} \right.
  \end{equation}
  where $\displaystyle f \in L^2(\Omega)$ and $\lambda>0$.   Following Definition \ref{defws},  a  function   $v \in L^1(\Omega)$ is said to be a weak solution of \eqref{st}  if  $v^m\in H^1_0(\Omega)$ and 
  	\begin{equation}\label{stwf}
  		\displaystyle  \int_\Omega v\:\xi+  \lambda \int_\Omega  \nabla v^m \cdot  \nabla\xi  - 
  		\lambda	\int_\Omega  v \:  V\cdot \nabla  \xi    =     \int_\Omega f\: \xi, \quad \hbox{ for all }	\xi\in H^1_0(\Omega).
  	\end{equation} 
  
  \begin{theorem} \label{texistm}
  Assume $V\in W^{1,2}(\Omega)$  and  $(\nabla \cdot V)^-\in L^\infty(\Omega) $.   For $f\in L^2(\Omega)$ and $\lambda$ satisfying 
  	\begin{equation}\label{condlambda1}
  		0<\lambda <  \lambda_0:= 1/\Vert (\nabla\cdot  V)^-\Vert_\infty ,  
  	\end{equation}
  	the problem \eqref{stwf} has a solution $v$ that we denote by $v_m.$  Moreover, 
  		for any $1\leq q\leq \infty,$ 
  		we have 
  		\begin{equation} \label{lqstat}
  		    \Vert v_m\Vert_q\leq   
  		    \left\{\begin{array}{ll} 
  		    		\Big( 1-  (q-1)   \lambda  /\lambda_0   \Big)^{-1}\Vert f\Vert_q, \quad &\hbox{ if }1\leq q<\infty \\ \\ 
  		    			\Big( 1-    \lambda/\lambda_0   \Big) ^{-1}  \Vert f\Vert_\infty, &\hbox{ if } q=\infty
  		    			\end{array}\right.
  		\end{equation} 
  		and  
  		\begin{equation}\label{lmst}
  			\Big( 1-\lambda/\lambda_0  \Big ) \int \vert v_m\vert ^{m+1}   \: dx + 	\lambda	\int \vert \nabla  v_m^m \vert^2\: dx 	 \leq  \int f\:  v_m^m \: dx    . 
  		\end{equation}
  		%
    \end{theorem}
  
  Moreover, thanks to Theorem \ref{tcompcmef},  we have 
  \begin{corollary}\label{ccomp}
	Under the assumption of Theorem \ref{texistm},  if moreover $V$ satisfies the outpointing  condition \eqref{HypV}, the problem 	\eqref{st} has a unique solution. Moreover, if $v _1$ and $v _2$ are   two   solutions associated with $f_1\in L^1(\Omega)$ and $f_2\in L^1(\Omega) $ respectively, then 
  	\begin{equation}
  		\label{ineqcomp}
  		\Vert (v_1-v_2)^+\Vert_{1} \leq \Vert (f_1-f_2)^+\Vert _{1} 
  	\end{equation}  
  	and 
  	$$ \Vert v_1-v_2\Vert_{1} \leq \Vert f_1-f_2\Vert _{1}.$$ 
  \end{corollary}
 \begin{proof}
	This is a simple consequence of the fact  that if $(v,p)$ (which is independent of $t$) is a solution of \eqref{st}, then it  can be assimilated to a time-independent solution of  the evolution problem \eqref{pmef} with $f$ replaced by $f-v$ (which is also independent of $t$).  \end{proof}

  \medskip 
  To prove Theorem \ref{texistm},  we proceed by regularization and compactness. For each $\eps >0,$ we consider   $\beta_\eps $    a regular Lipschitz continuous  function strictly increasing satisfying $\beta_\eps(0)=0$ and, as $\eps\to0,$  
  $$\beta_\eps(r)\to r^{1/m},\quad \hbox{ for any }r\in \RR.  $$ 
  One can take,  for instance,  $\beta_\eps$ the regularization by convolution of the application $r\in \RR\to r^{1/m}.$  Then, we  consider the problem 
  \begin{equation}	\label{pstbeta}
  	\left\{  \begin{array}{ll}\left.
  		\begin{array}{l}
  			\displaystyle v - \lambda\Delta p +  \lambda \nabla \cdot (v  \: V)=f \\   
  			\displaystyle v =\beta_\eps   (p) \end{array}\right\}
  		\quad  & \hbox{ in }  \Omega\\  \\  
  		\displaystyle p= 0  & \hbox{ on }  \partial \Omega  .\end{array} \right.
  \end{equation}

  \begin{lemma}\label{lexistreg}
  	For any $f\in L^2(\Omega)$ and $\eps>0,$ the problem 	\eqref{pstbeta} has a     solution $v_\eps ,$ in the sense that  $v_\eps \in L^2(\Omega),$ $p_\eps:=\beta_\eps^{-1} (u_\eps) \in H^1_0(\Omega),$        and 
  	\begin{equation}
  		\label{weakeps}  
  		\int v_\eps \: \xi \: dx + \lambda \int \nabla p_\eps \cdot \nabla \xi \: dx  -\lambda\int v_\eps \: V\cdot \nabla \xi \: dx =\int f\: \xi\: dx,   \end{equation}
  	for any $\xi\in H^1_0(\Omega).$ Moreover, for any $\lambda$ satisfying \eqref{condlambda1}
  	the solution $v_\eps$ satisfies the estimates  
  		\begin{equation} \label{lqstateps}
  		\Vert v_\eps \Vert_q\leq   
  		\left\{\begin{array}{ll} 
  			\Big( 1-  (q-1)   \lambda  /\lambda_0  \Big)^{-1}\Vert f\Vert_q, \quad &\hbox{ if }1\leq q<\infty \\ \\ 
  			\Big( 1-    \lambda/\lambda_0  \Big) ^{-1}  \Vert f\Vert_\infty, &\hbox{ if } q=\infty
  		\end{array}\right.
  	\end{equation}  	and  
  	\begin{equation} \label{lmsteps}
  		\Big( 1-\lambda/\lambda_0 \Big ) \int v_\eps\: p_\eps   \: dx + 	\lambda	\int \vert \nabla  p_\eps  \vert^2\: dx 	 \leq  \int f\: p_\eps \: dx    . 
  	\end{equation}
  \end{lemma}  
  \begin{proof}
  	We can assume without loss of generality throughout the proof that $\lambda=1$ and remove the script $\eps$ in the notations of $(v_\eps,p_\eps)$ and $\beta_\eps$.      We consider  $H^{-1} (\Omega) $ the usual topological dual space of $H^{1}_0(\Omega)$ and $ \langle .,. \rangle$ the associate  dual bracket.    See that the operator $A\: :\:  H^{1}_0(\Omega)\to  H^{-1} (\Omega) ,$ given by 
  	$$ \langle A p,\xi\rangle =  \int\beta  (p)\: \xi\: dx + \int \nabla p\cdot \nabla \xi \: dx -\int\beta  (p) \: V\cdot \nabla \xi \: dx,\quad \hbox{ for any }\xi,\: p\in H^1_0(\Omega), $$
  	is a bounded weakly continuous operator.  Moreover,  $A$ is coercive. Indeed, for any $p\in H^{1}_0(\Omega),$ we have 
  	\begin{eqnarray}
  		\langle A p,p\rangle &=& \int\beta  (p)\: p \: dx + \int \vert \nabla p\vert^2  \: dx -\int\beta  (p) \: V\cdot \nabla p  \: dx \\ 
  		&=&  \int\beta  (p)\: p \: dx  +   \int \vert \nabla p\vert^2  \: dx -\int   V\cdot \nabla\left(  \int_0^p  \beta  (r) dr\right)    \: dx \\   
  		&=&  \int\beta  (p)\: p \: dx  +   \int \vert \nabla p\vert^2  \: dx +\int  \nabla \cdot  V\:   \left(  \int_0^p  \beta  (r) dr\right)    \: dx  \\  
  		&\geq&  \int\beta  (p)\: p \: dx  +   \int \vert \nabla p\vert^2  \: dx - \int   {(\nabla \cdot   V)^{-}}  \:   p  \beta  (p)   \: dx  \\ 
  		&\geq&  \frac{1}{2}  \int\beta  (p)\: p \: dx  +      \int \vert \nabla p\vert^2  \: dx - \frac{1}{2}\int  {(\nabla \cdot   V)^{-}}^2   \: dx\\ 
  		&\geq&   \int \vert \nabla p\vert^2  \: dx - \frac{1}{2}\int   {(\nabla \cdot   V)^{-}}^2   \: dx ,  
  	\end{eqnarray}
  	where we use  Young inequality.   So, for any $f\in  H^{-1} (\Omega)  $ the problem $A p=f$ has a solution $p \in H^1_0(\Omega).$   
  	Now, for each  $1<q<\infty,$   taking  $   v ^{q-1}  $ as a test function, and  using the fact that 
  	$$v\nabla (  v ^{q-1}  )=   \frac{q-1}{q}\:  \nabla \vert v\vert^q ,\quad \hbox{ a.e. in }\Omega $$
  	and
  	$$ \nabla p\cdot \nabla  (  v ^{q-1}  )  \geq 0,  $$
  	we get 
  	\begin{eqnarray}
  		\int \vert v\vert^q  \: d x
  		&\leq& \int f\:   v ^{q-1} \: dx + \lambda \frac{q-1}{q}  \int  V\cdot  \nabla \vert v\vert ^q   \: dx 
  		\\ 
  		&\leq&  \int f  v ^{q-1 }  \: dx   -   \lambda \frac{q-1}{q}  \int  \nabla \cdot  V\:   \vert v\vert ^q   \: dx    	\\ 
  		&\leq&   \int f  v ^{q-1 }  \: dx   +   \lambda \frac{q-1}{q}  \int \left(  \nabla \cdot  V\right)^- \:   \vert v\vert ^q   \: dx         \\  
  		&\leq& \frac{1}{q} \int \vert f\vert^q   \: dx +    \frac{q-1}{q} \int \vert v\vert^q   \: dx   + \lambda \frac{q-1}{q} \left\Vert  \left(  \nabla \cdot  V\right)^- \right\Vert_\infty   \int \:   \vert v\vert ^q   \: dx, 
  	\end{eqnarray}
  	where we use     again   Young inequality. This implies  that 
  	\begin{equation}  
  		\left( 1-\lambda (q-1) \:   \Vert (\nabla \cdot V)^-\Vert_\infty  \right) \int \vert v\vert ^q\: dx \leq \int \vert f\vert ^q\: dx.
  	\end{equation}   
  	Thus  \eqref{lqstat}.   
  	To prove \eqref{lmst}, we take  $p$ as a test function, we obtain 
  	\begin{eqnarray}
  		\lambda	\int \vert \nabla p\vert^2\: dx &=&  \int fp\: dx -\int vp\: dx +\lambda \int \beta (p)V\cdot \nabla p\: dx \\
  		&=&  \int fp\: dx -\int vp\: dx +\lambda \int V\cdot \nabla  \left(  \int_0^p  \beta  (r) dr\right)  \: dx   \\ 
  		&=&  \int fp\: dx -\int vp\: dx -\lambda \int \nabla\cdot V  \:    \left(  \int_0^p  \beta  (r) dr\right) \: dx \\ 
  		&\leq& \int fp\: dx -\int vp\: dx + \lambda \int (\nabla \cdot V)^- \:   \int_0^ p\beta (r)dr \: dx    
  		\\ 
  		&\leq& \int fp\: dx -\int vp\: dx +\lambda\: \Vert (\nabla \cdot V)^- \Vert_\infty \: \int   vp  \: dx  
  	\end{eqnarray}
  	where we use  the fact that $\int_0^ p\beta (r)dr\leq p\beta (p) =vp  $.
  	Thus \eqref{lmst} for $1<q<\infty.$ For the case $q\in\{ 1,\: \infty\},$   we take  $H_\sigma (v-k)\in H^1_0(\Omega),$  for a given $k\geq 0$ and $\sigma>0,$     as a test function in  \eqref{pstbeta}, where 
  	\begin{equation} \label{Hsigma}
  		H_\sigma   (r) = \left\{ \begin{array}{ll} 
  			1 \quad &  \hbox{ if } r\geq 1 \\
  			r/\sigma & \hbox{ if }\vert r\vert <\sigma \\
  			-1 &\hbox{ if }r\leq -1 \: .	\end{array}\right.  
  	\end{equation} 
  	 Then, letting $\sigma \to 0$ and using  the fact that $ \nabla p\cdot \nabla  H_\sigma  (u-k)\geq 0$ a.e. in $\Omega,$   it is not difficult to see that 
  	\begin{equation}
  		\begin{array}{lll}
  			\int (v-k)^+ \:  dx & \leq&  \int  (f- k(1+\lambda \: \nabla \cdot V)) \:\sign_0^+(v-k)  \: dx+\lambda \lim_{\sigma \to 0} 
  			\int (v-k)\: V\cdot \nabla \hepsp (v-k)  \: dx \\ 
  			&\leq&  \int  (f- k(1+\lambda \: \nabla \cdot V)) \:\sign_0^+(v-k) \: dx,
  	\end{array}  \end{equation} 
  	where we use the fact that 
  	$\lim_{\sigma \to 0} 
  	\int (v-k)\: V\cdot \nabla H_\sigma (v-k)  \: dx=\lim_{\sigma \to 0} \frac{1}{\sigma}
  	\int_{0\leq v-k\leq \sigma} (v-k)\: V\cdot \nabla (v-k) \:   \: dx= 0.$ Thus 
  	$$ 	\int (v-k)^+ \:  dx \leq   \int  (f- k(1-\lambda \Vert (\nabla \cdot V)^-)) ^+ \: dx .$$ In particular, taking 
  		\begin{equation}
  		k=\frac{\Vert f\Vert_\infty}{1-\lambda/\lambda_0},
  	\end{equation}
  	we deduce that  $v\leq \frac{\Vert f\Vert_\infty}{1-\lambda /\lambda_0}.$  Working in the same way with $H_\sigma (-v+k)$ as a test function, we obtain  
  	$$  	v\geq - \frac{\Vert f\Vert_\infty}{1-\lambda/\lambda_0}.  $$
  	Thus the result of the lemma for $q=\infty.$ The case $q=1$ follows by Corollary \ref{ccomp}.  
  \end{proof}

  \begin{lemma}\label{lconveps}
  	Under the assumption of Theorem \ref{texistm}, by taking a subsequence $\eps\to 0$ if necessary, we have 
  	\begin{equation}\label{weakueps}
  		v_\eps\to v,\quad \hbox{ in }L^2(\Omega)\hbox{-weak}
  	\end{equation}  
  	and 
  	\begin{equation}\label{strongpeps}
  		p_\eps\to v^m,\quad \hbox{ in }H^1_0(\Omega). 
  	\end{equation}    
  	Moreover, $v$ is a weak solution of \eqref{st}. 
  \end{lemma}
  \begin{proof}  Using Lemma \ref{lexistreg} as well as Young and Poincar\'e inequalities, we see that the  sequences  $v_\eps$ and $p_\eps $  are bounded  in  $L^2(\Omega)$ and $H^1_0(\Omega),$ respectively.  So, there exists a subsequence that we denote again by $v_\eps$ and $p_\eps $   such that \eqref{weakueps} is fulfilled and 
  	\begin{equation}\label{weakpeps}
  		p_\eps\to v^m,\quad \hbox{ in }H^1_0(\Omega)\hbox{-weak}.
  	\end{equation}   
  	Letting $\eps\to 0$ in \eqref{weakeps}, we obtain that $v$ is a weak solution of \eqref{st}. Let us  prove that actually \eqref{weakpeps} holds to be true strongly in $H^1_0(\Omega).$ Indeed,  taking $p_\eps$ as a test function, we have 
  	\begin{equation}  
  		\begin{array}{ll} 
  			\lambda \int \vert \nabla p_\eps\vert^2\: dx  &=   \int   (f- v_\eps )\: p_\eps\: dx +\lambda  \int    V \cdot \nabla  \left(\int_0^{p_\eps}  \beta_\eps(r)dr\right) \: dx\\ 
  			&= \int   (f- v_\eps )\: p_\eps\: dx -\lambda    \int      \nabla \cdot   V   \:     \int_0^{p_\eps}  \beta_\eps(r)dr  \: dx    . 
  		\end{array}
  	\end{equation}     
  	Since $\int_0^r \beta_\eps(s)\: ds$ converges to  $\int_0^r \beta(s)\: ds,$ for any $r\in \RR,$ $p_\eps\to v^m$ a.e. in $\Omega$ and $\left\vert  \int_0^{p_\eps} \beta_\eps(s)\: ds\right\vert \leq v_\eps\: p_\eps$ which is bounded in $L^1(\Omega)$ by \eqref{lmsteps}, we have  
  	$$    \int_0^{p_\eps} \beta_\eps(s)\: ds \to     \int_0^{p} s^{1/m}\: ds= \frac{m}{m+1} \vert v\vert^{m+1},\quad \hbox{ in } L^1(\Omega). $$ 
  	So, in one hand we have  
  	\begin{equation} 
  		\begin{array}{ll}  	\lim_{\eps\to 0}  
  			\lambda 	\int \vert \nabla p_\eps\vert^2\: dx  &=  \int   (f- v )\: p\: dx -\lambda  \frac{m}{m+1}  \int      \nabla \cdot   V   \:       \vert v \vert^{m+1}   \: dx    .    	
  		\end{array}
  	\end{equation}  
  On the other, since $v$ is a weak solution of \eqref{st}, one sees easily that 
\begin{equation}
 \lambda \int \vert \nabla p\vert^2\: dx =  \int   (f- v )\: p\: dx -\lambda  \frac{m}{m+1}  \int      \nabla \cdot   V   \:       \vert v\vert^{m+1}   \: dx  \: ; 
\end{equation}  
which implies that  $ 	\lim_{\eps\to 0}  
   	\int \vert \nabla p_\eps\vert^2\: dx=   \int \vert \nabla p\vert^2\: dx.  $	Combing this with the weak convergence of $\nabla p_\eps,$ we deduce  the strong convergence \eqref{strongpeps}. 
  \end{proof}

  \begin{remark}
  	One sees in the  proof that the results of Lemma \ref{lconveps}  remain to be true if one replace $f$ in \eqref{pstbeta} by a sequence of $f_\eps\in L^2(\Omega)$ and assumes that, as $\eps\to 0,$  
  	$$f_\eps \to f,\quad \hbox{ in }L^2(\Omega).  $$
  \end{remark}
  
  \bigskip   
  \begin{proof}[\textbf{Proof of Theorem \ref{texistm}}]  
  	The proof follows by Lemma \ref{lconveps}.  Moreover, the estimates hold to be true by letting $\eps\to 0,$    in the    estimate \eqref{lqstateps} and \eqref{lmsteps} for $v_\eps$ and $p_\eps.$ 
  \end{proof}

 \medskip 
  
\subsection{Existence for the evolution problem } 

To study the evolution problem, we use Euler-implicit discretization  scheme. For any $n\in \NN^*$  being such that  $0<\eps:=T/n\leq \eps_0,$  we consider the sequence   $(u_i,p_i)_{i=0,...N}$ given by the   $\eps-$Euler implicit scheme associated with \eqref{pmef} :
\begin{equation}	\label{sti}
	\left\{  \begin{array}{ll}\left.
		\begin{array}{l}
			\displaystyle u_{i+1} - \eps \: \Delta p_{i+1}  +  \eps \: \nabla \cdot (u_{i+1}   \: V)=u_{i} +\eps  \: f_i  \\
			\displaystyle p_{i+1}  =u_{i+1}^m  \end{array}\right\}
		\quad  & \hbox{ in }  \Omega\\  \\
		\displaystyle p_{i+1} = 0  & \hbox{ on }  \partial \Omega ,\end{array} \right. \quad i=0,1, ... n-1,
\end{equation}
where,     $f_i$ is given by
$$f_i = \frac{1}{\eps } \int_{i\eps}^{(i+1)\eps }  f(s)\: ds ,\quad \hbox{ a.e. in }\Omega,\quad i=0,... n-1.  $$
Now, for a given $\eps-$time discretization $t_i=i\eps,$ $i=0,1,...n,$     we define the $\eps-$approximate solution by
\begin{equation}\label{epsapprox}
	u_\eps:=  \sum_{i=0} ^{n-1 }  u_i\chi_{[t_i,t_{i+1})},\quad \hbox{ and    }  \quad 	p_\eps:=   \sum_{i=1} ^{n -1}  p_i\chi_{[t_i,t_{i+1})}.
\end{equation}
In order to use the results of the previous section  and  the general theory of evolution problem governed by accretive operator (see for instance \cite{Benilan,Barbu}), we define the operator $\A_m$ in $L^1(\Omega),$ by $\mu\in \A_m(z)$ if and only if $\mu,\: z\in L^1(\Omega)$ and $z$ is a solution of the problem 
\begin{equation}	\label{eqopm}
	\left\{  \begin{array}{ll} 
			\displaystyle  - \Delta z^m +  \nabla \cdot (z  \: V)=\mu 
		\quad  & \hbox{ in }  \Omega\\   
		\displaystyle z= 0  & \hbox{ on }  \partial \Omega  ,\end{array} \right.
\end{equation}
in the sense that $z\in L^2(\Omega),$    $ z^m \in    H^1_0(\Omega)$  
and 
\begin{equation} 
	\displaystyle  \int_\Omega  \nabla z^m \cdot  \nabla\xi  -
	\int_\Omega  z \:  V\cdot \nabla  \xi    =     \int_\Omega \mu \: \xi, \quad \forall \:  \xi\in H^1_0(\Omega)\cap L^\infty(\Omega).
\end{equation}  
As a consequence of Theorem \ref{tcompcmef}, we see that  the  operator $\A_m$ is accretive in $L^1(\Omega)$ ; i.e. 
$(I+\lambda \A_m)^{-1}$ is a contraction in $L^1(\Omega),$ for small $\lambda>0$ (cf. Appendix section). Moreover, thanks to Theorem \ref{texistm},     we have

\begin{lemma}\label{lAm} For $0<\lambda<\lambda_0 ,$    $\overline{R(I+\lambda   \A_m)}=L^1(\Omega)$  and   $\overline {\D(\A_m)}=L^1(\Omega).$  
	\end{lemma}
\begin{proof}
	Since $R(I+\lambda   \A_m)\supseteq L^2(\Omega)$ (by Theorem \ref{texistm}), it is clear that $\overline{R(I+\lambda   \A_m)}=L^1(\Omega).$  To prove  the density of  ${\D(\A_m)}$ in $L^1(\Omega),$ we prove that $L^\infty(\Omega)\subseteq \overline {\D(\A_m)}.$  To this aim, for a given  $v\in L^\infty(\Omega),$ we consider the sequence $(v_\eps)_{\eps >0}$ given by the solution of the problem 
	  \begin{equation}	\label{st}
		\left\{  \begin{array}{ll} 
			\displaystyle v_\eps -\eps\Delta v_\eps^m + \eps   \nabla \cdot (v_\eps  \: V)=v 
			\quad  & \hbox{ in }  \Omega\\   
			\displaystyle v_\eps= 0  & \hbox{ on }  \partial \Omega.\end{array} \right.
	\end{equation}
Thanks to \eqref{lmst} and \eqref{lqstat}     with $q=\infty,$  one sees easily by letting $\eps-0$ that $v_\eps \to v$ in $L^1(\Omega)$-weak and then $\Vert v\Vert_1 \leq \lim_{\eps\to 0 }\Vert v_\eps\Vert_1.$  Moreover, thanks to  \eqref{lqstat}     with $q=1,$  we  deduce that  $\lim_{\eps\to 0 }\Vert v_\eps\Vert_1 = \Vert v\Vert_1.$  Thus $v_\eps \to v$ in $L^1(\Omega)$. 
\end{proof}

   So,  thanks to the general   theory of nonlinear semi-group governed by accretive  operator (see Appendix) , for any $u_0\in L^1(\Omega),$   we have       
\begin{equation}\label{convueps}
	u_\eps \to u,\quad \hbox{ in } \C([0,T),L^1(\Omega)),\quad \hbox{ as }\eps\to 0,
\end{equation}
where,   $u$ is the so called ''mild solution'' of the evolution problem
\begin{equation}\label{Cauchypb}
	\left\{\begin{array}{ll}
		u_t + \A_m u\ni f \quad & \hbox{ in }(0,T)\\  \\
		u(0)=u_0.
	\end{array}  \right.
\end{equation}
To accomplish the proof of existence (of weak solution)  for the problem  \eqref{pmef}, we prove that the mild solution $u$  satisfies all the conditions of Definition \ref{defws}.  More precisely, we prove the following result.

\begin{proposition}\label{pconv}
 Assume $V\in W^{1,2}(\Omega),$     $(\nabla \cdot V)^-\in L^\infty(\Omega) $ and $V$ satisfies the outpointing  condition \eqref{HypV1}.  	For any $u_0\in L^2(\Omega)$ and $f\in L^2(Q),$ the mild solution $u$ of the problem \eqref{Cauchypb} is the unique solution of \eqref{pmef}. 
\end{proposition}

To prove this result, thanks to \eqref{convueps}, it is enough to study  moreover   the limit of   sequence $p_\eps$ given by the $\eps-$approximate solution.


\begin{lemma}  Let $u_\eps$  and $p_\eps$ be  the $\eps-$approximate solution given by \eqref{epsapprox}. We have
	\begin{enumerate}
		\item For any $q\in [1,\infty],$ we have  
		  \begin{equation} \label{lquepsevol} 
			\Vert u_\eps(t)\Vert_q  \leq  M_q^\eps,\quad  \hbox{ for any }t\geq 0,   
		\end{equation}     
 \begin{equation} 
 	M_q^\eps := \left\{  \begin{array}{ll}
 		   \left( \Vert u_0 \Vert_q + \int_0^T   \Vert f_\eps (t)\Vert_q \: dt  \right) e^{  (q-1)\: T\:   \Vert (\nabla \cdot V)^-\Vert_\infty  } \quad & \hbox{ if } 1\leq q<\infty \\  \\  
 		    \left( \Vert u_0 \Vert_\infty + \int_0^T   \Vert f_\eps (t)\Vert_\infty \: dt  \right) e^{ T\:   \Vert (\nabla \cdot V)^-\Vert_\infty }   \quad & \hbox{ if } q=\infty  .
 		   \end{array}\right.
				\end{equation}   
		\item  For each $\eps >0,$  we have  
			\begin{equation}\label{lmuepsevol}
			\begin{array}{c} 
				\frac{1}{m+1}   \int_\Omega \vert u_\eps(t)\vert ^{m+1} +  \int_0^t\!\!  \int_\Omega \vert \nabla p_\eps\vert^2  \leq   \int_0^t\!\!     \int_\Omega f_\eps\:p_\eps \: dx +   	   \int_0^t\!\!   \int \left(  \nabla \cdot  V\right)^- \:    p_\eps \: u_\eps     \: dx \\   \\  + 	 \frac{1}{m+1}   \int_\Omega \vert u_{0}\vert ^{m+1}.  
			\end{array} 
		\end{equation} 	  
	\end{enumerate}	
\end{lemma}
\begin{proof}
	Thanks to  Theorem \ref{texistm}, the  sequence $  (u_i)_{i=1,...n}$ of solutions of \eqref{sti} is well defined in $L^2(\Omega) $  and satisfies  
	\begin{equation}\label{stwsi}
		\displaystyle  \int_\Omega u_{i+1} \:\xi+   \eps \: \int_\Omega  \nabla p_{i+1}  \cdot  \nabla\xi  -  \eps \:
		\int_\Omega  u_{i+1}  \:  V\cdot \nabla  \xi    =    \int_{i\eps  }^{(i+1)\eps  } \int_\Omega f_i\: \xi,\quad \hbox{ for }i=1,...,n-1,
	\end{equation} 
for any $\xi\in H^1_0(\Omega).$ 	Thanks to \eqref{lqstat}, for any $1\leq q\leq \infty,$ we have
\begin{eqnarray}
	 \Vert u_i\Vert_q &\leq&  	 \Vert u_{i-1}\Vert_q  + \eps\: \Vert f_{i}\Vert_q   +   \eps\: (q-1) \:   \Vert (\nabla \cdot V)^-\Vert_\infty  \Vert u_{i}\Vert_q .
  \end{eqnarray}
By induction, this implies that,  for any $t\in [0,T),$ we have  
  \begin{equation}  
  \Vert u_\eps(t)\Vert_q    \leq     \Vert u_0 \Vert_q    + \int_0^T \Vert f_\eps(t)\Vert_q \: dt    +  (q-1) \:   \Vert (\nabla \cdot V)^-\Vert_\infty     \int_0^T  \Vert u_\eps(t)\Vert_q\: dt  .
	\end{equation}   
Using Gronwall Lemma, we deduce  \eqref{lquepsevol}, for any $1\leq q< \infty.$   The proof for the case $q=\infty$ follows in the same way by using \eqref{lqstat} with $q=\infty.$ 
Now, using  the fact that   
	\begin{equation}
	 	\left(u_i-u_{i-1} \right) p_i = \left(u_i-u_{i-1} \right) u_i^m \geq   \frac{1}{m+1}\left  (u_i^{m+1} -u_{i-1}^{m+1}\right)    
	\end{equation}
and 
\begin{equation}
	\int u_i  \:  V \cdot \nabla p_i \leq \int \int\left(\nabla \cdot V\right)^- p_i\: u_i , 
\end{equation}	we get  
	\begin{eqnarray} 
		\frac{1}{m+1}   \int_\Omega \vert u_i\vert ^{m+1} + \eps \int_\Omega \vert \nabla p_i\vert^2 &\leq&   \eps  \int_\Omega f_i\:p_i \: dx +   	\eps  \:  \int \left(  \nabla \cdot  V\right)^- \:    p_i \: u_i     \: dx \\  &  &  + 	 \frac{1}{m+1}   \int_\Omega \vert u_{i-1}\vert ^{m+1}.
	\end{eqnarray} 
 Summing this identity for $i=1,....,$  and using the definition of $u_\eps$, $p_\eps$ and $f_\eps,$ we get \eqref{lmuepsevol}.
\end{proof}

\bigskip 
\begin{proof}[\textbf{Proof of Proposition \ref{pconv}}]  Recall that we already know that $u_\eps\to u$ in $\C([0,T);L^1(\Omega),$ as $\eps\to 0.$ Now, combining \eqref{lquepsevol} and \eqref{lmuepsevol} with Poincaré and Young  inequalities, one sees that 
\begin{equation}  	 
	\frac{1}{m+1}  \frac{d}{dt} \int_\Omega \vert u_\eps\vert ^{m+1} \: dx+  \int_\Omega \vert \nabla p_\eps\vert^2\: dx    \leq    C (N,\Omega)   \left(  	\int_\Omega \vert f_\eps\vert^2\: dx +	 \Vert (\nabla \cdot V)^-  \Vert_\infty\: (M_2^\eps)^2      \right ),\quad \hbox{ in }\D'(0,T).
\end{equation}   	
	This implies that 	$p_\eps$ is bounded in $L^2(0,T;H^1_0(\Omega)).$  This implies that 
 	\begin{equation}\label{convpeps}
		p_\eps \to u^m,\quad \hbox{ in }L^2(0,T;H^1_0(\Omega))-\hbox{weak},\quad \hbox{ as }\eps\to 0 .
	\end{equation}  
  Recall that taking 
	$$\tilde u_\eps(t) =\frac{(t-t_i)u_{i+1} -  (t-t_{i+1})u_i}{\eps} , \quad \hbox{ for any }t\in [t_i,t_{i+1}),\:  i=1,...n,$$
	we have 
\begin{equation}\label{uepstilde}
	\partial_t\tilde u_\eps -\Delta p_\eps +\nabla \cdot(u_\eps\: V)=f_\eps,\quad \hbox{ in }\D'(Q).
\end{equation}	
Moreover, we know that   $\tilde u_\eps \to u,$  in $\C([0,T),L^1(\Omega)).$ 
So letting $\eps\to 0 $ in \eqref{uepstilde},  we deduce that $u$ is a solution of \eqref{pmef}.   Letting $\eps \to 0$ in \eqref{lquepsevol} and \eqref{lmuepsevol}, we get respectively  \eqref{lquevol}
and   \eqref{lmuevol}.  
\end{proof}

 \bigskip 
 \begin{proof}[\textbf{Proof of Theorem \ref{texistevolm}}] The proof follows by Proposition \ref{pconv}.
   
\end{proof}

\section{The  limit as $m\to\infty.$}
\setcounter{equation}{0}

Since the solution of the problem \eqref{pmef} is the mild solution associated with the operator $\A_m,$  we begin by studying the $L^1-$ limit, as $m\to\infty,$ of the solution of the stationary problem  \eqref{st}.  Formally,   this limiting problem  is given by 
\begin{equation}	\label{sths}
	\left\{  \begin{array}{ll}\left.
		\begin{array}{l}
			\displaystyle v - \Delta p +  \nabla \cdot (v  \: V)=f \\
			\displaystyle v \in \hbox{Sign}(p)\end{array}\right\}
		\quad  & \hbox{ in }  \Omega\\  \\
		\displaystyle p= 0  & \hbox{ on }  \partial \Omega .\end{array} \right.
\end{equation}
This is the stationary problem associated with the so called Hele-Shaw problem. Thanks to \cite{Igshaw}, for any $f\in L^2(\Omega),$ \eqref{sths} has a unique solution $(v,p)$ in the sense  that    $(v ,p) \in  L^\infty(\Omega) \times H^1_0(\Omega)$,   $\displaystyle    v \in \sign(p)$ a.e. in $\Omega,$  and 
\begin{equation}\label{stwshs}
	\displaystyle  \int_\Omega v \:\xi+  \int_\Omega  \nabla p \cdot  \nabla\xi  -
	\int_\Omega  v \:  V\cdot \nabla  \xi    =     \int_\Omega f\: \xi, \quad \hbox{ for any } \xi\in H^1_0(\Omega).
\end{equation}

To begin with, we prove the following  incompressible limit for stationary problem.   

\begin{proposition}\label{tconvumWst}
	Under the assumptions of Theorem \ref{texistm}, let   $v_m$ be the solution of \eqref{st}. As $m\to\infty,$ we have 
	\begin{equation}\label{convum}
		v_m\to v,\quad \hbox{ in } L^2(\Omega)\hbox{-weak},
	\end{equation}
	\begin{equation}\label{convpm}
		v_m^m\to p,\quad \hbox{ in } H^1_0(\Omega),
	\end{equation} 
	and $(v,p)$ is the   solution of \eqref{sths}. 
\end{proposition} 
\begin{proof}  Thanks to  \eqref{lqstat}, there exists $v\in H^1_0(\Omega),$ such that    \eqref{convum} is fulfilled.  Thanks to \eqref{lmst}, we see that the  sequence  $p_m$  is   bounded   in $H^1_0(\Omega) ,$ which implies  that, by taking a sub-sequence if necessary,   
	\begin{equation}\label{weakpm}
		v_m^m\to p,\quad \hbox{ in } H^1_0(\Omega)\hbox{-weak} 
	\end{equation}  
and 
\begin{equation}\label{weakpm}
	v_m^m\to p,\quad \hbox{ in } L^2(\Omega).
\end{equation}  
	Using monotonicity arguments (see for instance Proposition 2.5 of \cite{Br}) we get   $ v \in \sign(p)$ a.e. in $\Omega, $ and  letting $m\to\infty$ in \eqref{weakeps}, we obtain that $(v,p)$ satisfies  \eqref{stwshs}.   To prove the strong convergence of $p_m,$ we use the same argument of the proof of Lemma \ref{lconveps}. Indeed,    taking $p_m$ as a test function in \eqref{stwf}, we have 
	\begin{eqnarray}  
		\lambda \int \vert \nabla p_m\vert^2\: dx  &=& \int   (f- v_m )\: p_m\: dx +\lambda   \int \nabla \cdot V \: \left(\int_0^{p_m}  r^\frac{1}{m}dr\right) \: dx\\  
		&=& \int   (f- v_m )\: p_m\: dx +\lambda  \frac{m}{m+1} \int \nabla \cdot V \: v_m\: p_m\: dx   . 
	\end{eqnarray}     
	Letting $m\to \infty,$ and using \eqref{convpm} and \eqref{weakpm},  we  see that 
	\begin{eqnarray} 
		\lim_{m\to \infty}  
		\lambda 	\int \vert \nabla p_m\vert^2\: dx  &=& \int  (f- u )\: p\: dx + \lambda \int u\: p\: \nabla \cdot V \:   dx \\ 
		&=& \int  (f- u )\: p\: dx + \lambda \int \nabla \cdot V \: \vert p\vert  \: dx . 
	\end{eqnarray}  
We know that $(u,p)$ is a solution   of \eqref{sths}, so one sees easily  that 
\begin{equation}
  \lambda \int \vert \nabla p\vert^2\: dx = \int  (f- u )\: p\: dx + \lambda \int \nabla \cdot V \: \vert p\vert  ,
\end{equation}
so that  
$$	\lim_{m\to \infty}  
\lambda 	\int \vert \nabla p_m\vert^2\: dx =  \lambda \int \vert \nabla p\vert^2\: dx.$$
	Thus the strong convergence of $\nabla p_m$. 
	
\end{proof}

For the strong convergence of $v_m,$ under the assumption \eqref{HypsupportV}, we prove first the following convergence for the solution of the stationary problem.

\begin{theorem}\label{tconvumSst}
	Under   the assumptions of Theorem \ref{tbvm} ;  i.e.  $V\in W^{1,2}(\Omega)$,  $\nabla \cdot V\in L^\infty(\Omega) $   and satisfies  \eqref{HypsupportV}, 	for any $0<\lambda < \lambda_1,$  the convergence \eqref{convum} holds to be true strongly in $L^1(\Omega)$.    Here
	\begin{equation}
		\lambda_1 :=   1/\sum_{i,k} \Vert \partial_{x_i} V_k \Vert_\infty. 
	\end{equation}

\end{theorem}

\begin{corollary}\label{cconvam}
Under the assumptions of Theorem \ref{tconvumSst}, the operator $\A_m$ converges to $\A$ in the sense of resolvent in $L^1(\Omega),$ where $\A$ is defined by  :  $\mu\in \A(z)$ if and only if $\mu,\: z\in L^1(\Omega)$ and $z$ is a solution of the problem 
	\begin{equation}	\label{eqopinfty}
		\left\{  \begin{array}{ll} 
			\displaystyle  - \Delta p +  \nabla \cdot (z  \: V)=\mu 
			\quad  & \hbox{ in }  \Omega\\    \\
			z\in \sign(p)\\  \\  
			\displaystyle p= 0  & \hbox{ on }  \partial \Omega  ,\end{array} \right.
	\end{equation}
	in the sense that $z\in L^\infty(\Omega),$  $\exists \: p\in H^1_0(\Omega)$ such that   $ p \in    H^1_0(\Omega),$  $u\in \sign(p)$ a.e. in $\Omega$   
	and 
	\begin{equation} 
		\displaystyle  \int_\Omega  \nabla p  \cdot  \nabla\xi  -
		\int_\Omega  z \:  V\cdot \nabla  \xi    =     \int_\Omega \mu \: \xi, \quad \forall \:  \xi\in H^1_0(\Omega)\cap L^\infty(\Omega).
	\end{equation}   
Moreover, we have 
$$\overline{\D(\A)} =\Big\{ z\in L^\infty(\Omega)\: :\:  \vert z\vert \leq 1\hbox{ a.e. in }\Omega \Big\}.  $$
\end{corollary}

The main element  to prove Theorem \ref{tconvumSst}  is  $BV_{loc}$-estimates on $v_m.$  Recall that a given function   $u\in L^1(\Omega)$ is said to be of bounded variation if and only if, for each $i=1,...N,$    
\begin{equation}
	TV_i(u,\Omega)  := \sup\left\{ 	\int_\Omega  u\: \partial_{x_i} \xi \: dx  \:  :\:  \xi\in \mathcal C^1_c(\Omega)  \hbox{ and }\Vert \xi\Vert_\infty \leq 1\right\}  <\infty, \end{equation}
here   $\C^1_c(\Omega)$ denotes the set of $\C^1-$function compactly supported in $\Omega.$   More generally a function is locally of bounded variation in a domain $\Omega$ if and only if for any open set $\omega\subset\!   \subset \Omega,$   $ TV_i(u,\omega)<\infty$ for any $i=1,...,N.$ 
In general a function  locally of bounded variation (as well as function of bounded variation) in $\Omega,$  may not be differentiable, but by the Riesz representation theorem, their  partial derivatives  in the sense of distributions are  Borel measure in $\Omega.$  This gives rise to the definition of the   vector space of functions of   bounded variation in $\Omega$, usually denoted by    $BV(\Omega),$ as the set of   $u\in L^1(\Omega)$ for which there   are Radon measures $\mu_1,...,\mu_N$ with finite total mass in $\Omega$ such that   
\begin{equation}
	\int_\Omega v\: \partial_{x_i} \xi \: dx =-\int_\Omega \xi\: d\mu_i,\quad \hbox{ for any }\xi\in \mathcal C_{c}(\Omega),\quad \hbox{ for }i=1,...,N. 
\end{equation}
Without abuse of notation  we continue to point out the measures $\mu_i$ by $\partial_{x_i}v$ anyway,  and by $\vert \partial_{x_i}v\vert $   the total variation of $\mu_i.$ Moreover, we'll use  as usual $Dv=(\partial_{x_1}v,...,\partial_{x_N}v)$  the vector valued Radon measure  pointing out the gradient of any function $v\in BV(\Omega),$ and $\vert Dv\vert$ indicates the total variation measure of  $v.$   In particular, for any open set $\omega\subset\!   \subset \Omega,$ $TV_i(v,\omega)=\vert \partial_{x_i}v\vert(\omega)<\infty,$ and the total variation of the function $v$ in  $\omega$  is finite too ; i.e. 
\begin{equation}
	\Vert Dv\Vert(\omega)= \sup\left\{ 	\int_\Omega  v\: \nabla \xi \: dx  \:  :\:  \xi\in \mathcal C^1_c(\omega)  \hbox{ and }\Vert \xi\Vert_\infty \leq 1\right\}  <\infty.
\end{equation}
At last, let us remind the reader here  the well known  compactness result for functions of bounded variation : given a sequence $u_n$ of functions in $BV_{loc}(\Omega)$ such that, for any   open set $\omega\subset\!   \subset  \Omega,$ we have 
\begin{equation}
	\sup_n\left\{ \int_\omega \vert v_n\vert\: dx + \vert Dv_n\vert (\omega) \right\} <\infty,
\end{equation} 
there exists  a subsequence that we denote again by $v_n$  which converges in $L^1_{loc}(\Omega)$ to a function   $v\in BV_{loc}(\Omega).$  Moreover, for any  compactly supported continuous function $0\leq \xi$, the limit $u$ satisfies 
\begin{equation}
	\int \xi\:  \vert \partial_{x_i}v \vert \leq \liminf_{n\to\infty }  \int \xi\: \vert \partial_{x_i}v_n \vert ,  
\end{equation}
for any $i=1,...N,$ and 
\begin{equation}
	\int \xi\:  \vert Dv\vert \leq \liminf_{n\to\infty }  \int \xi\:  \vert Dv_n\vert.
\end{equation}

 \bigskip
 Under the assumption \eqref{HypsupportV}, the following sequence of test functions plays an important role in the proof of  $BV_{loc}$-estimates and   convergence results of Theorem  \ref{treglimum}.

\begin{lemma}\label{lomegah}
	Under the assumption  \eqref{HypsupportV},  	  there exists $0\leq \omega_h\in \mathcal H^2(\Omega_h)$ compactly supported in $\Omega,$ such that $\omega_h \equiv 1$ in  $    \Omega_h$ and 
	\begin{equation}  \label{HypsupportV2}  
		\	\int_{\Omega\setminus \Omega_h}  \varphi\:  V\cdot \nabla \omega_h  \: dx \geq  0,\quad \hbox{ for any }0\leq \varphi\in L^2(\Omega),
	\end{equation}
\end{lemma}
\begin{proof} It is enough to take $\omega_h(x)=\eta_h(d(.,\partial \Omega)),$ for any $x\in \Omega,$ where $\eta_h\:  :\:   [0,\infty)\to \RR^+$ is  a  nondecreasing $\C^2$-function compactly supported in $(0,\infty)$ such that   $\eta_h \equiv 1$  in $[	h,\infty).$ In this case, 
	$$\nabla \omega_h =\eta_h'(d(.,\partial \Omega))\: \nabla d(.,\partial \Omega), $$ 
	so that 
	\begin{eqnarray*}   
		\	\int_{\Omega\setminus \Omega_h}  \varphi\:  V\cdot \nabla \omega_h  \: dx &=&  - \int_{\Omega\setminus \Omega_h} \eta_h'(d(.,\partial \Omega))\:   \varphi\:  V\cdot \nabla d(.,\partial \Omega)   \: dx \\  \\ 
		&=&   \geq  0,\quad \hbox{ for any }0\leq \varphi\in L^2(\Omega),
	\end{eqnarray*}
	This function may be define  
	\begin{equation} 
				\eta_h(r)= \left\{  \begin{array}{ll}
			0\quad &\hbox{ if } 0\leq r\leq c_1h\\  \\ 
						e^{\frac{-C_h}{r^2-c_1^2h^2 }}\quad &\hbox{ if } c_1 h\leq r\leq  c_2h\\  \\ 
								1-	e^{\frac{-C_h}{h^2 -r^2}} \quad &\hbox{ if } c_2h\leq r\leq h\\  \\ 
												1\quad &\hbox{ if }  h \leq r 	\end{array}  \right.
	\end{equation} 
	with $0<c_1<c_2<1$ and $C_h>0$ given such that $2c_2^2-c_1^2=1$ and $e^{\frac{-C_h}{M_h}} = 1-e^{\frac{-C_h}{M_h}},$ 
	where $M_h:= (c_2^2-c_1^2)h^2 =  (1-c_2^2)h^2.$   For instance one can take $c_1=1/2$ and $c_2=  \sqrt{5}/(2\sqrt{2}).$  See that 
		\begin{equation} 
		\eta_h'(r)= \left\{  \begin{array}{ll}
			0\quad &\hbox{ if } 0\leq r\leq c_1h\\  \\ 
			\frac{2rC_h}{(r^2-c_1^2h^2)^2}e^{\frac{-C_h}{r^2-c_1^2h^2 }}\quad &\hbox{ if } c_1 h\leq r\leq  c_2h\\  \\ 
		 \frac{2rC_h}{(h^2-r^2)^2} 	e^{\frac{-C_h}{h^2 -r^2}} \quad &\hbox{ if } c_2h\leq r\leq h\\  \\ 
			0\quad &\hbox{ if }  h \leq r 	\end{array}  \right.
	\end{equation} 
is continuous and derivable at least on $\RR\setminus \{ c_2h\}.$ Thus $\eta_h\in H^2(\Omega).$

\end{proof}
 \medskip 
 
 We have 
\begin{theorem}  \label{tbvm}
	Assume    $f\in BV_{loc}(\Omega)\cap L^2(\Omega)$,   $V\in W^{1,\infty}(\Omega)^N,$      $\nabla \cdot V\in W^{1,2 }_{loc}(\Omega)  $ and let $v_m$ be the solution of \eqref{st}. Then, for any $0<\lambda <1/\lambda_1$,  $v_m\in BV_{loc}(\Omega)$ and   we have 
	\begin{equation} \label{bvstat} 
		\begin{array}{c}
			(1-\lambda \lambda_1 ) 	\sum_{i=1}^N \int     \omega_h  \:  d\: \vert \partial_{x_i}v\vert   \leq  \lambda	\sum_{i=1}^N    \int (\Delta \omega_h)^+  \:  \vert \partial_{x_i} p \vert   \: dx   +  	\sum_{i=1}^N \int       \omega_h \: d\:  \vert \partial_{x_i} f\vert      \\  + \lambda   \sum_{i=1}^N   	  \int        \omega_h \:  \vert v\vert \:   \vert  \partial_{x_i}   ( \nabla \cdot  V   )\vert     \: dx     ,
		\end{array}
	\end{equation} 
where $\omega_h$ is given by Lemma \ref{lomegah}. 
\end{theorem}

To prove this result we use again the regularized problem  \eqref{pstbeta} and we  let $\eps\to 0.$  To begin with, we prove first the following lemma concerning any weak solutions of the general problem 
\begin{equation}\label{eqalpha}	  
	\displaystyle  v -  \Delta \beta^{-1} (v)+     \nabla \cdot (v  \: V)=f  \hbox{ in } \Omega, 
\end{equation}
$\beta$  is a given a nondecreasing function assumed to be regular (at least $\C^2$).

\begin{lemma} \label{lbvreg}
Assume     $f\in W^{1,2}_{loc}(\Omega)$, $V\in W^{1,2}_{loc}(\Omega)^N$,   $\nabla \cdot V\in W^{1,\infty }_{loc}(\Omega)  $ and  $v\in H^1_{loc}(\Omega)\cap L^\infty_{loc}(\Omega)$ satisfy  \eqref{eqalpha} in $\D'(\Omega).$  Then,   for each $i=1,..N,$ we have    
	\begin{equation}\label{bvestreg} 
		\begin{array}{c}   
			\vert \partial_{x_i}	 v \vert - 
		 	\sum_{k=1}^N  \vert    \partial_{x_k} v\vert   \:    \sum_{k=1}^N  \vert       \partial_{x_i} V_k  \vert  -    \Delta 	\vert \partial_{x_i}  \beta^{-1} (v)  \vert  +      \nabla \cdot( \vert \partial_{x_i}  v\vert \: V)   \\   \leq \vert 	\partial_{x_i} f \vert 
			+   \vert v\vert \:\vert  \partial_{x_i}(\nabla \cdot  V) \vert  \quad \hbox{ in } \D'(\Omega).
		\end{array} 
	\end{equation}
\end{lemma}
\begin{proof}  Set $p:=\beta^{-1}(v).$  Thanks to   \eqref{eqalpha} and the regularity of $f$ and $V,$ it is not difficult to see that   $v,\: p\in H^2_{loc}(\Omega)\cap L^\infty_{loc}(\Omega),$   and for  each $i=1,...N,$  the partial derivatives $\partial_{x_i}v$ and $\partial_{x_i} p$ satisfy the following equation  
	\begin{equation}\label{bvint1}
		\partial_{x_i}	 v -  \Delta 	\partial_{x_i} p + 	    \nabla \cdot(\partial_{x_i}v\: V  )=	\partial_{x_i} f -  ( \nabla v\cdot   \partial_{x_i} V  +  v\: \partial_{x_i}  (\nabla \cdot  V)),\quad \hbox{ in } \D'(\Omega).
	\end{equation}  
	By density, we can take  $\xi H_\sigma(\partial_{x_i}v) $ as a test function in \eqref{bvint1} where $\xi \in H^2(\Omega)$  is    compactly supported  in $\Omega$ and  $H_\sigma$   is given by \eqref{Hsigma}. We obtain 
	\begin{equation}\label{bvint2}
		\begin{array}{cc} 	\int \Big( 	\partial_{x_i}	 v \:  \xi H_\sigma(\partial_{x_i}v)  +\nabla 	\partial_{x_i} p   \cdot \nabla (\xi H_\sigma(\partial_{x_i}v) \Big)  \: dx - \int 	 \partial_{x_i}  v\: V \cdot \nabla (\xi H_\sigma (\partial_{x_i}v)) \: dx  \\  \\      =	\int \partial_{x_i} f\: \xi H_\sigma (\partial_{x_i}v) \: dx -   \int    ( \nabla v\cdot   \partial_{x_i} V  +  v\: \partial_{x_i}  (\nabla \cdot  V) ) \:  \xi H_\sigma (\partial_{x_i}v)\: dx  . 
	\end{array} 	\end{equation} 
	To pass to the limit as $\sigma\to 0,$  we see first  that
	 \begin{equation} \label{triv}
	 H_\sigma '(\partial_{x_i}v)\: \partial_{x_i}v= \frac{1}{\sigma} \: \partial_{x_i}v  \: \chi_{[\vert \partial_{x_i} v_\eps\vert \leq \sigma ]}\:  \to 0 ,\quad \hbox{ in }L^\infty(\Omega)\hbox{-weak}^*.
 	 \end{equation}
	So,   the last term of the first part of \eqref{bvint2} satisfies  
	\begin{eqnarray}
		\lim_{\sigma \to 0}	 \int 	 \partial_{x_i}  v\: V \cdot \nabla (\xi H_\sigma (\partial_{x_i}v)) \: dx &=&     \int 	 \vert \partial_{x_i}  v\vert \: V \cdot \nabla  \xi  \: dx +  \lim_{\sigma  \to 0}	 \int  \partial_{x_i}v  \: \nabla   \partial_{x_i}v  \cdot  V   H_\sigma '(\partial_{x_i}v) \: \xi  \: dx \\  
		&=&   \int 	 \vert \partial_{x_i}  v\vert \: V \cdot \nabla  \xi, 
	\end{eqnarray}
 On the other hand, we see that 
	\begin{eqnarray}
		\int \nabla 	\partial_{x_i} p   \cdot \nabla (\xi H_\sigma (\partial_{x_i}v) ) \: dx  &=&\int  H_\sigma (\partial_{x_i}v)  	\nabla 	\partial_{x_i} p   \cdot \nabla \xi   \: dx  + \int \xi\: 	\nabla 	\partial_{x_i} p   \cdot \nabla  H_\sigma (\partial_{x_i}v) \: dx .  
	\end{eqnarray}
	Since $\sign_0 (\partial_{x_i}v)=\sign_0 (\partial_{x_i}  p),$   the first term satisfies
	\begin{equation}
		\lim_{\sigma\to 0} \int  H_\sigma (\partial_{x_i}v)  	\nabla 	\partial_{x_i} p   \cdot \nabla \xi   \: dx  =-   \int   \vert \partial_{x_i} p\vert  \: \Delta  \xi   \: dx.  
	\end{equation}  
	As to the second term, we have 
	\begin{eqnarray}
		\lim_{\sigma\to 0} 	\int \xi\: 	\nabla 	\partial_{x_i} p   \cdot \nabla  H_\sigma (\partial_{x_i}v) \: dx &=&   \lim_{\sigma\to 0} 	\int \xi\: H_\sigma '(\partial_{x_i}v)\: 	\nabla 	\partial_{x_i} p   \cdot \nabla \partial_{x_i}v    \: dx \\  
		&=&   \lim_{\sigma\to 0} 	\int \xi\: H_\sigma '(\partial_{x_i}v)\: 	\nabla   (\beta'(v)\partial_{x_i}v)   \cdot \nabla \partial_{x_i}v    \: dx\\  
		&=&   \lim_{\sigma\to 0} 	\int \xi\: H_\sigma '(\partial_{x_i}v)\: 	\beta'(v) \: \Vert \nabla  \partial_{x_i}v\Vert^2   \: dx   \\   &  & +    \lim_{\sigma\to 0} 	\int \xi\: H_\sigma '(\partial_{x_i}v)\: \partial_{x_i}v \: \beta''(v) 	\nabla  v    \cdot \nabla \partial_{x_i}v    \: dx  \\  
		&\geq &  \lim_{\sigma\to 0} 	\int \xi\: H_\sigma '(\partial_{x_i}v)\: \partial_{x_i}v \: \beta''(v) 	\nabla  v    \cdot \nabla \partial_{x_i}v    \: dx \\   &\geq &0, 
	\end{eqnarray}
 where we use again \eqref{triv}.  
	So,  letting   $\sigma\to 0$ in \eqref{bvint2} and  	using  again the fact that   $\sign_0 (\partial_{x_i}v)=\sign_0 (\partial_{x_i}  p),$        
	we get    
	$$ \begin{array}{c}   
		\vert \partial_{x_i}	 v \vert -  \Delta 	\vert \partial_{x_i} p \vert  +    \nabla \cdot( \vert \partial_{x_i}  v\vert \: V)     \leq \sign_0 (\partial_{x_i}  v) 	\partial_{x_i} f    -           ( \nabla v\cdot   \partial_{x_i} V  \\   \\   +  v\: \partial_{x_i}(\nabla \cdot  V) ) \: \sign_0 (\partial_{x_i}v) \quad \hbox{ in } \D'(\Omega) 
	\end{array}  $$  
At last, using the fact that 
	$$\vert \nabla v\cdot   \partial_{x_i} V\vert \leq   \sum_{k}  \vert    \partial_{x_k} v\vert   \:    \sum_{k}  \vert       \partial_{x_i} V_k  \vert  ,$$ 
	the result of the lemma follows.   
\end{proof}

\bigskip
\begin{proof}[\textbf{Proof of Theorem \ref{tbvm}}] 	Under the assumptions of Theorem  \ref{tbvm}, for any $\epsilon >0,$ let us consider $f_\eps$  a  regularization   of $f$  satisfying $f_\eps \to f$ in $L^1(\Omega)$ and 
	\begin{equation}
		\int\xi\: 	\vert \partial_{x_i} f_\eps\vert \: dx \to  \int\xi\: d\: 	\vert \partial_{x_i} f\vert , \quad   \hbox{ for any } \xi\in \C_c(\Omega)  \hbox{ and }  i=1,...N.
	\end{equation} 	
		Thanks to Lemma \ref{lexistreg},  we consider $v_\eps$ be the solution of the problem \eqref{st}, where we replace  $f$ by the regularization  $f_\eps$.    Applying   Lemma \ref{lbvreg} by replacing  $V$ by $\lambda V$ and $\beta^{-1}$ by $\lambda\: \beta_\eps^{-1}$,   we obtain  
	\begin{equation} 
		\begin{array}{c}
			\int      \vert \partial_{x_i}v_\eps \vert     \: \xi\: dx     -  \lambda     \int    \sum_{k=1}^N   \vert       \partial_{x_i} V_k  \vert  \:       \sum_{k=1}^N  \vert \partial_{x_k}v_\eps\vert    \: \xi\: dx   
				\leq  \lambda   \sum_{k=1}^N  \int 	\vert \partial_{x_i} p_\eps \vert  \: (\Delta \xi)^+\:    dx + 
				\int \vert 	\partial_{x_i} f_\eps \vert \: \xi \: dx   \\      +   	 \lambda  \int    \vert    
			v_\eps\vert \:  \vert \partial_{x_i} ( \nabla \cdot     V  ) \vert 	 \: \xi \: dx     -  \lambda   \:    \int     \vert \partial_{x_i}v_\eps   \vert\:    V\cdot \nabla \: \xi\: dx   	,\quad \hbox{ for any }i=1,...N \hbox{ and }0\leq \xi\in \D(\Omega).
		\end{array}
	\end{equation}   
By density, we can take $\xi=\omega_h$ as given by Lemma \ref{lomegah}, so that the last term   is nonnegative, and we have 
	\begin{equation} 
		\begin{array}{c}
			\int      \vert \partial_{x_i}v \vert     \: \omega_h\: dx     -  \lambda        \:    \sum_{k}  \vert       \partial_{x_i} V_k  \vert  \:    \int  \sum_{k} \vert \partial_{x_k}v\vert    \: \omega_h\: dx      	\leq  \lambda  \sum_{k}   \int 	\vert \partial_{x_i} p \vert  \: (\Delta  \omega_h)^+\:    dx    \\   + 
			\int \vert 	\partial_{x_i} f  \vert \: \omega_h \: dx   +  	 \lambda  \int    \vert    
			v\: \vert \:  \vert \partial_{x_i} ( \nabla \cdot     V  ) \vert 	 \: \xi_h  \: dx   	,\quad \hbox{ for any }i=1,...N.
		\end{array}
	\end{equation}   
 	Summing up,  for $i=1,...N,$  and using the definition of $\lambda_1,$ we deduce that  
	\begin{equation} 
		\begin{array}{c}
			\sum_{i} 		\int      \vert \partial_{x_i} v_\eps \vert     \: \omega_h\: dx     -  \lambda \lambda_1   \sum_{k}   \int  \vert \partial_{x_k}v_\eps\vert    \: \xi_h\: dx   	\leq  \lambda  \sum_{i}   \int  	(\Delta \omega_h)^+\: \vert \partial_{x_i} p_\eps \vert  \: dx    \\    + 
			\int\sum_{i}  \vert 	\partial_{x_i} f_\eps  \vert \: \omega_h \: dx    + \lambda \:  \int   \vert v_\eps\vert \sum_{i}      \vert \partial_{x_i} (\nabla \cdot   V  )  \vert  \: \omega_h \: dx  ,
		\end{array} 
	\end{equation}   
	and then the corresponding property \eqref{bvstat} follows for $v_\eps.$ 
	Thanks to  \eqref{lqstat} and \eqref{lmst}, we know that $v_\eps$ and $\partial_{x_i}p_\eps$ are bounded in $L^2(\Omega).$ This implies that, for any $\omega\subset\!\subset \Omega,$  $\sum_{i} 		\int_\omega      \vert \partial_{x_i} v_\eps \vert     \:   dx$ is bounded.  So,  $v_\eps$ is bounded in $BV_{loc}(\Omega).$  Combining this with the $L^1-$bound \eqref{lqstat}, it implies  in particular,  taking a subsequence if necessary,     the convergence in \eqref{weakueps} holds to be true also in $L^1(\Omega)$ and   then $v\in BV_{loc}(\Omega).$   At last, letting $\eps\to\infty$ in \eqref{bvestreg} and, using moreover \eqref{strongpeps} and the lower semi-continuity of variation measures   $\vert \partial_{x_i} v_\eps\vert$, we deduce  
	\eqref{bvstat} for the limit $v,$ which is the solution of the problem \eqref{st} by Lemma \ref{lconveps}.
\end{proof}

 \bigskip
\begin{proof}[\textbf{Proof of Theorem \ref{tconvumSst}}]
	Recall that under the assumptions of the theorem, the $BV_{loc}$ estimate \eqref{bvstat} is fulfilled for $v_m.$   Since the constant $C$ in  \eqref{lmst} does not depend on $m,$ this implies  that  $v_m$ is bounded in $BV(\omega).$ Since $\omega$ is arbitrary,  we deduce in particular   that  the convergence in \eqref{convum} holds to be true also in $L^1(\Omega)$,   $v\in BV_{loc}(\Omega),$ and  
	 \eqref{bvstat} is fulfilled.  
\end{proof}

 \bigskip
\begin{proof}[\textbf{Proof of Theorem \ref{treglimum}}]
Thanks to Corollary \ref{cconvam} and Theorem \ref{trcontinuity} , we have   
\begin{equation}
	u_m\to u,\quad \hbox{ in } \C([0,T);L^1(\Omega)). 
\end{equation}
On the other hand, thanks to \eqref{lqstat} and \eqref{lmst}, it is clear that $p_m$ is bounded in $L^2(0,T;H^1_0(\Omega)).$ So, there exists $p\in L^2(0,T;H^1_0(\Omega)),$ such that, taking a subsequence if necessary, we have 
\begin{equation}
	u_m^m\to p,\quad \hbox{ in }L^2(0,T;H^1_0(\Omega))-\hbox{weak}.
\end{equation}
Then using monotonicity arguments we have $u\in \sign(p)$ a.e. in $Q,$ and letting $m\to\infty,$ in the weak formulation we deduce that the couple $(u,p)$ satisfies \eqref{weakformhs}. Thus the results of the theorem. 
 
\end{proof}

 \begin{remark}\label{Rbvcond}
 	See that we use the condition \eqref{HypsupportV} for the proof of $BV_{loc}$-estimate through $w_h$ we introduce  in Lemma \ref{lomegah}.  Indeed, $BV_{loc}$ estimates follows  from Lemma \ref{lbvreg}  once  the term $  \lambda   \:    \int     \vert \partial_{x_i}v_\eps   \vert\:    V\cdot \nabla \: \xi\: dx$ nonegative. 
 	Clearly, the construction of  $\omega_h$ by using $d(.,\partial \Omega)$ is  basically connected to the condition \eqref{HypsupportV}.  Otherwise, this condition could be replaced definitely by the existence  of  $0\leq \omega_h\in \mathcal H^2(\Omega_h)$ compactly supported in $\Omega,$ such that $\omega_h \equiv 1$ in  $    \Omega_h$ and 
 	\begin{equation}  \label{HypsupportV2}  
 		\	\int_{\Omega\setminus \Omega_h}  \varphi\:  V\cdot \nabla \omega_h  \: dx \geq  0,\quad \hbox{ for any }0\leq \varphi\in L^2(\Omega).
 	\end{equation}
 \end{remark}

\section{Reaction-diffusion case}\label{Sreaction}

Let us consider now  the reaction-diffusion porous medium equation   with linear drift  
\begin{equation} 	\label{pmeg}
	\left\{  \begin{array}{ll} 
		\displaystyle \frac{\partial u }{\partial t}  -\Delta u^m +\nabla \cdot (u  \: V)=g(.,u)  \quad  & \hbox{ in } Q \  \\  
		\displaystyle u= 0  & \hbox{ on }\Sigma  \\  \\   
		\displaystyle  u (0)=u _0 &\hbox{ in }\Omega,\end{array} \right.
\end{equation}
 Thanks to Theorem \ref{abstractRevol} and Theorem \ref{trcontinuity}, we assume that   $g\: :\: Q\times\RR\rightarrow\RR$   is a Carathéodory application ; i.e.  continuous in $r\in\RR$ and measurable in  $(t,x)\in Q$, and satisfies moreover  the following assumptions :

\begin{itemize}
	\item [($\G_1)$]  $g(.,r)\in L^2(Q) $  for any $r\in \RR.$ 
	
	\item [($\G_2)$]   There exists $0\leq \theta \in \C(\RR) ,$ such that 
	$$\frac{\partial g}{\partial r}(t,x,.) \leq \theta  ,\quad  \hbox{ in }\D'(\RR),\hbox{  for a.e. }(t,x)\in Q.$$

		\item [($\G_3)$]  There exists  $\omega_1,\: \omega_2  \in W^{1,\infty}(0,T)$ such that $w_1(0)\leq u_0\leq w_2(0)$  a.e. in $\Omega$ and,   for any $t\in (0,T),$  
		\begin{equation}\label{inf}
			\dot \omega_1(t)+ \omega_1(t)\nabla \cdot V \leq  g(t,.,\omega_1(t))\quad \hbox{ a.e. in  }\Omega
		\end{equation} 
	and
		\begin{equation}\label{sup}
			\dot \omega_2(t)+ \omega_2(t)\nabla \cdot V \geq  g(t,.,\omega_2(t))\quad \hbox{ a.e. in  }\Omega.
		\end{equation}
\end{itemize}

\begin{remark}\label{remg}
\begin{enumerate}
	\item 	On sees that  $(G2)$ implies that, $g(.,u)\in L^1(Q),$ for any $u\in L^\infty(Q).$ Indeed,   setting $M=\int_0^{\Vert u\Vert_\infty}\theta (r)\: dr,$ we have 
		\begin{equation}\label{explicitcondg} 
		-	g^-(.,M) - 2M\max_{[-M,M]}\theta  \leq g(.,u(.))\leq g^+(.,-\Vert u\Vert_\infty) + 2M\max_{[-M,M]}\theta  ,\quad \hbox{ a.e. in }Q,.
	\end{equation} 

\item As we will see, the main achievement  of the condition $(\G_3)$ is  some kind of  à priori $L^\infty$ estimates.  These conditions are accomplish in many practical situations.  For instance they are fulfilled in the case where $g(.,r)\in L^\infty(Q),$ for any $r\in \RR,$ and there exists $w\geq 0 $  solution of the following autonomous ODE 
	\begin{equation}\label{inf1}
	\dot \omega= \omega\Vert (\nabla \cdot V)^-\Vert_\infty + \Vert g(.,w)\Vert_\infty \quad \hbox{ in  }(0,T),\quad \hbox{ and } \omega(0)=\Vert u_0\Vert_\infty. 
\end{equation} 
 Actually, in this case it is enough to take $w_2(t)=-w_1(t)=w(t),$ for any $t\in [0,T).$ This is fulfilled   for instance in the case where the application $r\in \RR\to  \Vert g(.,r)\Vert_\infty $ is locally Lipschitz and $ \Vert g(.,0)\Vert_\infty =0.$  However, one needs to be careful with the choice of $T$ to fit it on with the maximal time for the solution of  the  ODE above.  This may generates local (and not necessary global) existence of a solution even if $g(t,x,r)$ is well defined  for any $t\geq 0.$  
\item Particular example for $g$  may be given as follows : 
\begin{enumerate}
	\item \label{remg2a}If $g(.,r)=f(.)$,  a.e. in $Q$ and  for any $r\in \RR,$ where $f\in L^\infty(Q),$ then it enough to take 
	$$w_2(t)=-w_1(t)=  \left( \Vert u_0\Vert_\infty + \int_0^t  \Vert f(t)\Vert_\infty \right)e^{t\Vert (\nabla \cdot V)^-\Vert_\infty},\quad \hbox{ for any }t\in (0,T) .  $$

\item If $g(.,r)=f(.)\: r$, a.e. in $Q$ and for any $r\in \RR,$ where $f\in L^\infty(Q),$ then it enough to take 
$$w_2(t)=-w_1(t) =   \Vert u_0\Vert_\infty e^{t\Vert (\nabla \cdot V)^-\Vert_\infty+ \int_0^t  \Vert f(t)\Vert_\infty },\quad \hbox{ for any }t\in (0,T)  . $$

\item If $\nabla \cdot V\geq 0$ and $g(t,x,r)= r^2,$ one can take   $w_1(t)=\frac{\Vert u_0\Vert }{1-t\: \Vert u_0\Vert} $ and $w_2(t)=\frac{-\Vert u_0\Vert }{1+t\: \Vert u_0\Vert} $, for $i=1,2.$  But, in this case $T$ needs to be taken   such that $T\leq 1/\Vert u_0\Vert_\infty.$   
	\end{enumerate}

 \item Thanks to the  remarks above, one sees that $\nabla \cdot V^+$  is less  involved  in the existence of a solution than $g$ and $\nabla \cdot V^-$.   

\end{enumerate}
\end{remark}

\bigskip
\begin{theorem}\label{texistg}
Assume $u_0\in L^2(\Omega)$ and $V\in W^{1,2}(\Omega)$ is such that   $\nabla \cdot V\in L^\infty(\Omega)$ and   satisfies the outpointing  condition \eqref{HypV0}. 	Under the assumption $(\G_1)$,  $(\G_2)$ and  $(\G_3)$,     the problem  \eqref{pmeg} has a unique weak solution $u_m$  in the sense of Definition \ref{defws}  with $f=g(.,u).$    Moreover,   we have 
	\begin{enumerate}
		\item $u$ is the unique mild solution of  the Cauchy problem \eqref{Cauchypb} 
		with $f(.)= g(.,u(.))$  a.e. in $Q.$ 
		
		\item for any $0\leq t<T,$ $\omega_1(t) \leq 	u(t)\leq \omega_2(t) $ a.e. in $\Omega.$ 
	\end{enumerate}  	 
\end{theorem}

\begin{remark}
	See that in the case where, for any $r\in \RR,$ $g(.,t)=f$ a.e. in $Q,$ with $f\in L^\infty(Q)$ we retrieve the $L^\infty$-estimate \eqref{lquevol} for $q=\infty.$ Indeed, thanks to Theorem \ref{texistg}  and  Remark  \ref{remg} (see the item 3-(a)), we see that 
	\begin{eqnarray*}
	\Vert u(t)\Vert_\infty &\leq&  	 \left( \Vert u_0\Vert_\infty + \int_0^t  \Vert f(t)\Vert_\infty \right)e^{t\Vert (\nabla \cdot V)^-\Vert_\infty}, \quad \hbox{ for any } t\in (0,T) \\  \\
	&\leq& \left( \Vert u_0\Vert_\infty + \int_0^T  \Vert f(t)\Vert_\infty \right)e^{T\Vert (\nabla \cdot V)^-\Vert_\infty} = M_\infty .  
	\end{eqnarray*}
	\end{remark}

\begin{corollary} \label{cexistg} Assume $0\leq u_0\in L^2(\Omega)$ and $V\in W^{1,2}(\Omega)$ is such that   $\nabla \cdot V\in L^\infty(\Omega)$ and   satisfies the outpointing  condition \eqref{HypV0}. 	If 
		\begin{equation}
		( \G_4) \quad  0 \leq g(.,0)  \hbox{  a.e. in }
		Q .\end{equation} 
	and, there   exists  $\omega  \in W^{1,\infty}(0,T)$ such that $0\leq u_0\leq w(0)$  a.e. in $\Omega$ and for any $t\in (0,T),$  
	\begin{equation}\label{inf}
		\dot \omega(t)+ \omega(t)\nabla \cdot V \leq  g(t,.,\omega(t))\quad \hbox{ a.e. in  }Q
	\end{equation} 
	then    the solution of \eqref{pmeg} satisfies 
	\begin{equation}\label{solpos}
 0\leq u(t)\leq \omega_2(t) ,\quad \hbox{  a.e. in }\Omega,\hbox{ for any }t\in (0,T).
	\end{equation}   
\end{corollary}
\begin{proof}
This is a simple consequence of Theorem \ref{texistg} where we take  $w_1\equiv 0$.
\end{proof}

\begin{remark}
	A typical example of $g$ satisfying the assumption  of Corollary \ref{cexistg} may be given by   $g(.,r)=\mu(r),$ a.e.  in $Q,$ for any $r\geq 0,$ with  $0\leq \mu\in \C(\RR^+),$  and   there exists $w \geq 0 $ such that  
	$$  \Vert \nabla \cdot V \Vert_\infty \leq  \frac{\mu(w)}{w}  .$$ 
\end{remark}

\bigskip 
\begin{proof}[\textbf{Proof of Theorem \ref{texistg}}]
Let  $F\: :\: [0,T)\times L^1(\Omega)\to L^1(\Omega)$ be given    by
	$$F(t,z(.))=   g(t, .,(z(.) \vee (-M)) \wedge M )  \quad \hbox{ a.e. in }\Omega, \hbox{ for any }(t,z)\in [0,T)\times L^1(\Omega),  $$
	where $M:=\max(\Vert \omega_1\Vert_\infty,\Vert \omega_2\Vert_\infty) .$ 
	Thanks to Remark \ref{remg}, one sees that $F$ satisfies all the assumptions of Theorem \ref{abstractRevol}.     Then,   thanks to Theorem \ref{Crandall-Liggett},  we consider $u\in \C([0,T),L^1(\Omega))$   the mild solution of    the evolution problem
	\begin{equation}
		\left\{\begin{array}{ll}
			u_t + \A_m u\ni  F(.,u)\quad & \hbox{ in }(0,T)\\  \\
			u(0)=u_0.
		\end{array}  \right.
	\end{equation}
   Thanks to \eqref{explicitcondg},  it is clear that
	$ F(.,u)  \in L^2(Q),$ so that, using  Proposition  \ref{pconv},  we can deduce that  $u$ is a weak solution of \eqref{pmef}.    The uniqueness follows from the equivalence between weak solution and mild solution as well as the uniqueness result of Theorem  \ref{abstractRevol}.  To end up the proof, it is enough to show that    $\omega_1(t) \leq 	u(t)\leq \omega_2(t) $ a.e. in $\Omega,$   for any $0\leq t<T.$ Indeed, in particular this implies that  $F(t,u(t))=g(t,.,u(t)),$ and the proof of existence is complete.   To this aim, we use  Theorem \ref{tcompcmef}  with the  the fact that $\omega_2$ is a weak solution of  \eqref{pmef} 	with  $f= \dot \omega_2+ \omega_2\:  \nabla \cdot V$,  to see that 
		\begin{eqnarray}
			\frac{d}{dt} \int (u-\omega_2)^+ \: dx &\leq & \int_{[u\geq \omega_2 ]} (g(.,u)-  \dot \omega_2- \omega_2 \nabla \cdot V ) \: dx\\ 
			&\leq&      \int_{[u\geq \omega_2 ]}   ( g(.,u) -g(.,\omega_2)    )  \: dx \\
			&\leq&    \max_{[\omega_1,\omega_2]}\theta\:   \int (u-\omega_2)^+ \: dx.
		\end{eqnarray}
		Applying Gronwall Lemma and using the fact that $u(0)\leq \omega_2(0),$ we obtain $u(t)\leq \omega_2$ a.e. in $Q.$  The proof of $u\geq  \omega_1$ in $Q$ follows in the same way by
		proving that
		$$ 	\frac{d}{dt} \int (\omega_1-u)^+ \leq    \max_{[\omega_1,\omega_2]}\theta\:    \int (\omega_1-u)^+ .$$
 Thus the results of the theorem. 
	
\end{proof}

\bigskip 
Now, for the limit of the solution of \eqref{pmeg}, thanks to Theorem \ref{tconvumSst} and Theorem \ref{trcontinuity},  we have the following result. 
\begin{theorem}
Assume $V\in W^{1,2}(\Omega), $ $\nabla \cdot V\in L^\infty(\Omega)$ and $V$ satisfies the outpointing  condition \eqref{HypsupportV}.   Let $g_m$ be a sequence of Carathéodory applications satisfying $(\G_1)$,  $(\G_2)$ and  $(\G_3)$with $\theta$ independent of $m.$ For any $u_{0m}\in L^2(\Omega)$ being  a sequence of initial data let   $u_m$ be   the sequence of corresponding solution of \eqref{pmeg}. If 
 \begin{equation}\label{convgm}
 	g_m(.,r)\to g(.,r),\quad \hbox{ in }L^1(Q),\quad \hbox{ for any }r\in \RR,
 \end{equation}
 and 
  \begin{equation}\label{convu0m}
 u_{0m}\to u_0,\quad \hbox{ in }L^1(\Omega), \quad \hbox{ and }\vert u_0\vert \leq 1\hbox{ a.e. in }\Omega,
 \end{equation}
 then, we have 
 \begin{enumerate}
 	\item $u_m\to u$ in $\C([0,T),L^1(\Omega))$
 	\item $u_m^m\to p$ in $L^2(0,T;H^1_0(\Omega))$-weak
 	\item $(u,p)$ is the solution of the Hele-Shaw problem 
 	
 	\begin{equation}  \label{hlsg}
 		\left\{  \begin{array}{ll} 
 		\left. \begin{array}{l}
 				\displaystyle \frac{\partial u }{\partial t}  -\Delta p +\nabla \cdot (u  \: V)=g(.,u) \\  
 				u\in \sign(p)\end{array}\right\}  \quad  & \hbox{ in } Q \  \\   \\ 
 			\displaystyle u= 0  & \hbox{ on }\Sigma  \\  \\   
 			\displaystyle  u (0)=u _0 &\hbox{ in }\Omega,\end{array} \right.
  	\end{equation}
   in the sense  that $(u,p)$ is the   solution  of \eqref{evolhs0} with $f(.)=g(.,u(.))$ a.e. in $Q$ satisfying $u(0)=u_0.$ 
 \end{enumerate}
 \end{theorem}
\begin{proof}
To begin with we prove compactness of $u_m$ in $\C([0,T);L^1(\Omega)).$ We know that $u_m$ is the mild solution of the sequence of Cauchy problems   
\begin{equation}
	\left\{\begin{array}{ll}
		u_t + \A_m u\ni  F_m(.,u) \quad & \hbox{ in }(0,T)\\  \\
		u(0)=u_{0m},
	\end{array}  \right.
\end{equation}
where, for a.e. $t\in (0,T)$, $F_m(t,z) = g_m(t,.,z(.))\vee (-M)) \wedge M )),$ a.e. in $\Omega,$  for any $z\in L^1(\Omega),$  and 
  $$ M:=\max( \Vert \omega_1\Vert_\infty,  \Vert \omega_2\Vert_\infty) . $$ 

Thanks to  \eqref{explicitcondg}, one sees that $F_m$ satisfies all   the assumptions  of Theorem \ref{trcontinuity}. 
This implies, by Theorem \ref{trcontinuity}, that 
\begin{equation}\label{convum2}
	u_m\to u,\quad \hbox{ in }\C([0,T);L^1(\Omega)), \hbox{ as }m\to\infty.
\end{equation}
Thus the compactness of $u_m.$ On the other hand, remember that $u_m$ is a weak solution of 
 \begin{equation}  
	\left\{  \begin{array}{ll} 
		\displaystyle \frac{\partial u }{\partial t}  -\Delta u^m +\nabla \cdot (u  \: V)=f_m  \quad  & \hbox{ in } Q \\  \\  
		\displaystyle u= 0  & \hbox{ on }\Sigma  \\  \\   
		\displaystyle  u (0)=u _{0m} &\hbox{ in }\Omega.
	\end{array} \right.
\end{equation}
with $f_m:=g(.,u_m).$   Using again \eqref{explicitcondg}, \eqref{convgm} and \eqref{convum2},   we see that  
 $$f_m\to g(.,u)\quad \hbox{ in }L^1(Q),\quad \hbox{ as }m\to\infty.$$
   So, by  Theorem \ref{treglimum}, we deduce that $u_m^m\to p$ in    $L^2(0,T;H^1_0(\Omega))$-weak and 
 $(u,p)$ is the solution of the Hele-Shaw problem  \eqref{hlsg}.    At last the uniqueness follows from the $L^1-$comparison  results of the solutions of the Hele-Shaw problem \eqref{evolhs0} (cf. \cite{Igshuniq}) as well as the assumption  ($\G_2)$   (one can see also \cite{IgshuniqR} for more details on Reaction-Diffusion Hele-Shaw flow  with linear drift).  
 \end{proof}

\section{Appendix}
\subsection{Reminder on evolution problem governed by accretive operator}
\setcounter{equation}{0}

Our aim here is to remind the reader on some basic tools on $L^1-$nonlinear semi-group theory we use in this paper. We are interested   in PDE which can be   be written in the following form 
\begin{equation}\label{abstractevol}
	\left\{  \begin{array}{ll}
		\frac{du}{dt}  +B u \ni f \quad & \hbox{ in } (0,T)\\ \\ 
		u(0)=u_{0},
	\end{array}\right.
\end{equation}	  
where $B$ is a possibly multivalued operator defined on  $L^1(\Omega)$ by its graph 
$$B=\left\{ (x,y)\in L^1(\Omega)\times L^1(\Omega)\: :\: y\in Bx\right\},$$ $f\in L^1(0,T;L^1(\Omega))$ and $u_0\in L^1(\Omega).$ An operator  $B$ is said to be accretif  in $L^1(\Omega)$ if and only if the operator $J_\lambda := (I+\lambda\:  B)^{-1}$ defines a contraction in $L^1(\Omega),$ for any $\lambda >0$ ;  i.e.   if for $i=1,2,$  $(f_i-u_i)\in \lambda Bu_i,$ then $\Vert u_1-u_2\Vert_1 \leq \Vert f_1-f_2\Vert_1 .$  

To study the evolution problem \eqref{abstractevol} in the framework of nonlinear semi-group theory in the Banach space $L^1(\Omega),$  the main ingredient is to use the operator $J_\lambda,$ through the Euler-Implicit time discretization scheme.  For an arbitrary $n\in \NN^*$ such that $0<\eps:=T/n\leq \eps_0,$    we consider the sequence of $(u_i,p_i)_{i=0,...n}$ given by   :
\begin{equation} 
	u_i+\eps B u_i\ni \eps f_i +u_{i-1},\quad \hbox{ for }i=1,...n,
\end{equation}
where,   for each $i=0,...n-1,$ $f_i$ is given by
$$f_i = \frac{1}{\eps } \int_{i\eps}^{(i+1)\eps }  f(s)\: ds ,\quad \hbox{ a.e. in }\Omega.  $$
Then, for a given $\eps-$time discretization $t_i=i\eps,$ $i=0,...n,$     we define the $\eps-$approximate solution  
\begin{equation}
	u_\eps:=  \sum_{i=0} ^{n-1 }  u_i\chi_{[t_i,t_{i+1})},
\end{equation}
and  its linear interpolate given by 
\begin{equation}
	\tilde u_\eps(t) = \sum_{i=0} ^{n-1 }  \frac{(t-t_{i})u_{i+1} -  (t-t_{i+1})u_{i} }{t_{i+1} -t_{i} } \: \chi_{[t_i,t_{i+1})}(t) ,\quad \hbox{ for any }t\in [0,T).
\end{equation} 
In particular, one sees that $u_\eps,$ $\tilde u_\eps$ and $f_\eps$ satisfies the following $\eps-$approximate dynamic 
\begin{equation}
	\frac{d\tilde u_\eps}{dt} +Bu_\eps \ni f_\eps,\quad \hbox{ in }(0,T).
\end{equation}
The main goal afterwards is to let $\eps\to 0,$ to cover the ''natural'' solution of  the Cauchy problem \eqref{abstractevol}. The following theorem known as Crandall-Liggett theorem (at least in the case where $f\equiv 0,$ cf.  \cite{CrLi}) pictures the limit of $u_\eps$ and $\tilde u_\eps.$

\begin{theorem}\label{Crandall-Liggett}
	Let $B$ be an accretive operator in $L^1(\Omega)$ and $u_0\in \overline{D(B)}.$ 	If for each  $\eps>0,$ the  $\eps-$approximate solution $u_\eps$   is well defined, then  there exists a unique $u\in \C([0,T),L^1(\Omega))$ such that $u(0)=u_0,$  
	\begin{equation}
		u_\eps \to u\quad \hbox{ and } \quad \tilde{u}_\eps \to u \quad \hbox{ in }\C([0,T),L^1(\Omega)),  \hbox{ as }\eps\to 0.
	\end{equation}  
	The function $u$ is called the mild solution of the evolution problem \eqref{abstractevol}.  Moreover, if $u_1$ and $u_2$ are two mild solutions associated with $f_1$ and $f_2,$ then there exists $\kappa\in L^\infty(\Omega),$ such that $\kappa\in \sign(u_1-u_2)$   a.e.  in $Q,$ and   
	\begin{equation} 
		\frac{d}{dt} \Vert u_1-u_2\Vert_1 \leq \int_{[u_1=u_2]}  \vert  f_1-f_2\vert \:  dx +\int_{[u_1\neq u_2]} \kappa\: (f_1-f_2) \: dx ,\quad \hbox{ in }\D'(0,T). 
	\end{equation}  
\end{theorem}

 \medskip 
 
 On sees that this theorem figures out  in a natural way a solution to the Cauchy problem \eqref{abstractevol}  to settle existence and uniqueness  questions for the associate PDE. However, in general we do not know in which sense the limit $u$ satisfies the concluding PDE ; this is connected to the regularity of $u$ as well as to the compactness of  $	\frac{d\tilde u_\eps}{dt}.$ We refer interested readers to \cite{Barbu} and \cite{Benilan} for more developments and examples in this direction. One can see also the book \cite{Br} in the case of Hilbert space, for which the concept  of  accretive operator   is appointed by   monotone graph notion.

 One sees that besides the accretivity (monotinicity in the case of Hilbert space) the well  posedness for the ''generic'' associate stationary problem  
 $$u+\lambda \: Bu\ni g,\quad \hbox{ for a given }g  $$
 is first need. Thereby, a sufficient condition for the results of Theorem \ref{Crandall-Liggett} is given by the so called range condition  
 \begin{equation}
	\overline{\R(I+\lambda B)} =L^1(\Omega),\quad \hbox{ for small  } \lambda>0 . 
\end{equation}
Indeed, in this case Euler-Implicit time discretization scheme  is well pose for any $i=0,...n-1,$ and the $\eps-$approximate solution is well defined (for small $\eps >0$).  Then the convergence to unique mild solution  $u$ follows by accretivity (monotinicity in the case of Hilbert space).

   In particular, Theorem \ref{Crandall-Liggett}    enables to associate to each accretif  operator $B$ satisfying   the   range condition    a nonlinear semi-group of contraction in $L^1(\Omega).$  It  is given by   Crandall-Ligget exponential formula  
\begin{equation}
	e^{-tB} u_0=  L^1-\lim \left(  I+\frac{t}{n} B\right)^{-n} u_0, \quad \hbox{ for any }u_0\in \overline{\D(B)}.
\end{equation}
In other words   the mild solution of \eqref{abstractevol} with $f\equiv 0$ is given by $e^{-tB}u_0.$

 \bigskip
The attendance   of a reaction in nonlinear PDE hints to study evolution problem of the type  
 	\begin{equation}\label{abstractevolF}
 	\left\{ 	\begin{array}{ll}
 		\displaystyle\frac{du}{dt}+Bu\ni  F(.,u)  \quad &  \hbox{ in }  (0,T)\\  \\
 		\displaystyle u(0)=u_0 ,
 	\end{array}\right.
 \end{equation} 
where $F\: :\: (0,T)\times L^1(\Omega)\to L^1(\Omega),$  is assumed to be Carathéodory, i.e. $ F(t,z)$  is measurable in $t\in (0,T)$ and  continuous in $z\in L^1(\Omega).$  To solve the evolution problem \eqref{abstractevolF} in the framework of $\eps-$approximate/mild solution, we say that $u\in \C([0,T);L^1(\Omega))$ is a mild solution of   \eqref{abstractevolF} if and only if 
$u$ is a mild solution of  \eqref{abstractevol} with $f(t)=F(t,u(t))$ for a.e. $t\in (0,T).$ Existence and uniqueness   are more or  less well known in the case where  $F(t,r)=f(t)+F_0(r)$, with $f(t)\in L^1(\Omega),$ for a.e. $t\in [0,T),$  and $F_0$   a Lipschitz continuous function in $\RR$.    The following theorems set up general assumptions on $F$ to ensure  existence and uniqueness of mild solution for \eqref{abstractevolF}, as well as continous dependence with respect to $u_0$ and $F.$ We refer the readers to \cite{BeIgsing} for  the detailed of proofs in abstract Banach spaces.

\bigskip
To call back these results, we  assume moreover that $F$ satisfies    the following assumptions : 
	\begin{itemize}
	
	\item[$(F_1)$]  There exists  $k\in L^1_{loc}  (0,T) $ such that
	\begin{equation} 
		\int  ( F(t,z)-F(t,\hat z))\: \so(z-\hat z)  \: dx    \leq  k(t)\: \Vert z-\hat z\Vert_1, \quad \hbox{ a.e. }t\in (0,T), 
	\end{equation}
	for every $z,\ \hat z\in \overline{D(B)} .$
	\item[$(F_2)$]  There exists  $c\in L^1_{loc}  (0,T) $ such that
	\begin{equation}
		\Vert F(t,z)\Vert_1 \leq c(t), \quad \hbox{ a.e. }t\in (0,T)
	\end{equation}
	for every $z\in \overline{D(B)} .$
\end{itemize}
In particular, one sees that under these  assumptions, $F(.,u)\in L^1_{loc}(0,T;L^1(\Omega)) $ for any $u\in$ $\C([0,T);L^1(\Omega)).$

\begin{theorem}\label{abstractRevol} (cf. \cite{BeIgsing}) 
	If $B$ be an accretive operator in $L^1 (\Omega)$ such that $J_\lambda$ well defined in a dense subset of $L^1(\Omega),$ then,  for any $u_0\in\overline {D(B)} $  there exists a unique mild solution $u$ of \eqref{abstractevolF} ; i.e. $u$ is the unique function in $\C([0,T);X)$, s.t.   $u$ is the    mild solution of
	\begin{equation}
		\left\{ 	\begin{array}{ll}
			\displaystyle\frac{du}{dt}+Bu\ni f \quad &  \hbox{ in }  (0,T)\\  \\
			\displaystyle u(0)=u_0,
		\end{array}\right.\end{equation}
	with $f(t)=F(t,u(t))$ a.e. $t\in (0,T).$
\end{theorem}

\bigskip\bigskip 
Another important results concerns the continuous dependence of the solution with respect to the operator $B$ as well to the data $f_n$ and $u_{0n}$ is given in the following theorem. The proof may be found in \cite{BeIgsing}.

\begin{theorem}\label{trcontinuity} (cf. \cite{BeIgsing})  
	For $m=1,2, ...  ,$ let $B_m$ be an accretive  operators in $L^1(\Omega)$ satisfying the range condition and    $F_m\ :\  (0,T)\times  \overline \D(B_m ) \to L^1 (\Omega)$ a  Carathéodory applications satisfying  $(F_1)$ and $(F_2)$  with $k$ and $c$ independent of $m.$  For each $m=1,2,...  $ we consider    $u_{0m} \in \overline \D( B_m)$ and $u_m$ the mild solution of the evolution problem 
	\begin{equation}
		\left\{ 	\begin{array}{ll}
			\displaystyle\frac{du}{dt}+B_mu\ni f_m \quad &  \hbox{ in }  (0,T)\\  \\
			\displaystyle u(0)=u_{0m},
		\end{array}\right.\end{equation}
	with $f_m=F_m(.,u)$.
	If,  there exists an accretive  operators $B$  in $L^1(\Omega)$  and a Carathéodory $F\ :\  (0,T)\times  \overline {D(B)} \to L^1 (\Omega)$  such that \begin{itemize}
		\item[a)] 	$(I+\lambda B_m)^{-1} \to (I+\lambda B)^{-1}\quad \hbox{ in }L^1(\Omega),  \hbox{ for any }0<\lambda <\lambda_0$
		\item[b)] $F_m(t,z_m) \to F(t,z)$ in $L^1(\Omega),$ for a.e. $t\in (0,T),$  and for any $z_m\in  \overline{D(B_m)}$ such that $\lim_{m\to \infty } z_m=z\in\overline {D(B)}. $ 
		\item[c)] there exists $u_0\in \overline {D(B)},$ such that $ u_{0m}\to  u_{0},$    
	\end{itemize} 
	then 
	\begin{equation} \label{convumabstract}
		u_m\to u, \quad \hbox{ in } \C([0,T),L^1(\Omega)),
	\end{equation}
	and $u $ is the unique mild solution of 
	\begin{equation}
		\left\{  \begin{array}{ll}
			\frac{du}{dt}  +B  u \ni F(.,u)\quad & \hbox{ in } (0,T)\\ \\ 
			u(0)=u_{0}.
		\end{array}\right.
	\end{equation}		 
\end{theorem}

\end{document}